\documentclass[]{interact}

\theoremstyle{plain}
\newtheorem{theorem}{Theorem}[section]
\newtheorem{lemma}[theorem]{Lemma}
\newtheorem{corollary}[theorem]{Corollary}
\newtheorem{proposition}[theorem]{Proposition}

\theoremstyle{definition}
\newtheorem{definition}[theorem]{Definition}
\newtheorem{example}[theorem]{Example}

\theoremstyle{remark}
\newtheorem{remark}{Remark}

\newtheorem{chained-construction}{{\bf Chained one-dimensional double extension construction}}

	\usepackage[utf8]{inputenc}
	\usepackage[T1]{fontenc}
	\usepackage[english]{babel}
	\usepackage{amsmath, amsthm, amsfonts, amssymb}
	\usepackage{bm, enumerate, graphicx, array, xcolor}
	\usepackage{cite, multirow}
	
\newcommand{\fn}{\mathfrak{n}}
\newcommand{\fa}{\mathfrak{a}}
\newcommand{\fL}{\mathfrak{L}}
\newcommand{\ku}{\mathbb{K}}
\newcommand\blank{{\mkern 2mu\cdot\mkern 2mu}}

\DeclareMathOperator{\cdim}{codim}
\DeclareMathOperator{\ad}{ad} 
\DeclareMathOperator{\der}{Der} 
 
\DeclareMathOperator{\spa}{span} 
\DeclareMathOperator{\im}{Im} 
\DeclareMathOperator{\inner}{Inner}

\DeclareMathOperator{\alt}{Alt}
\DeclareMathOperator{\sgn}{sgn}
\DeclareMathOperator{\rank}{rank}
\DeclareMathOperator{\rad}{Rad}

\let\oldforall\forall
\renewcommand{\forall}{\oldforall \, }

\let\oldexist\exists
\renewcommand{\exists}{\oldexist \: }

\begin{document}

\title{Equivalent constructions of nilpotent quadratic Lie algebras}

\author{
	\name{Pilar Benito\textsuperscript{a}\thanks{CONTACT Pilar Benito. Email: pilar.benito@unirioja.es.} and Jorge Rold\'an-L\'opez\textsuperscript{a}}
	\affil{\textsuperscript{a}Dpto. de Matem\'aticas y Computaci\'on, Universidad de La Rioja, Logro\~no, La Rioja, Spain}
}

\maketitle
\begin{abstract}
	The double extension and the $T^*$-extension are classical methods for constructing finite dimensional quadratic Lie algebras. The first one gives an inductive classification in characteristic zero, while the latest produces quadratic non-associative algebras (not only Lie) out of arbitrary ones in characteristic different from $2$. The classification of quadratic nilpotent Lie algebras can also be reduced to the study of free nilpotent Lie algebras and their invariant forms. In this work we will establish an equivalent characterization among these three construction methods. This equivalence reduces the classification of quadratic $2$-step nilpotent to that of trivectors in a natural way. In addition, theoretical results will provide simple rules for switching among them.
\end{abstract}


\section{Introduction}
Quadratic Lie algebras, also named as metrisable, were introduced in 1957 (see~\cite{Tsou_Walker_1957}) as real Lie algebras of Lie groups admitting a  Riemannian metric invariant under all translations of the group. In fact, according to \cite[Lemmas 7.1 and 7.2]{milnor1976curvatures} (see also \cite[Lemma 2.1]{medina1985groupes}), the connected Lie groups admitting a bi-invariant Riemannian metric are those Lie groups for which their Lie algebras are quadratic. In \cite{Tsou_Walker_1957} several decomposition and existence theorems are given, and it is shown that every metrisable algebra decomposes as an orthogonal sum of an abelian algebra and a finite number of non-decomposable reduced ones. The family of quadratic Lie algebras is quite large and contains reductive Lie algebras, and also infinitely many solvable examples. Their structure has essential patterns which can be used to decode the structure of some Lie groups. Riemannian Geometry makes this class of algebras visible, but they also play an important role in many other branches of mathematics and physics from Cartan's Criterion up to completely integrable Hamiltonian systems (see \cite[Section 1]{bordemann1997nondegenerate}).

A bilinear form $\varphi$ on a non-associative algebra $A$ with product $xy$ that satisfies,
\begin{equation}\label{eq:invariante}
	\varphi(xy,z)=\varphi(x,yz)\quad \forall x,y,z\in A,
\end{equation}
is called invariant. If the invariant bilinear form $\varphi$ is non-degenerated, the pair $(A,\varphi)$ is named \emph{pseudo-quadratic} algebra and \emph{quadratic} if, in addition, $\varphi$ is symmetric. In the literature they appear also named as, metric, metrised, metrisable (usual names for algebras over the real field), orthogonal, regular quadratic, quassi-classical or symmetric self-dual.

Double extensions and $T^*$-extensions are classical construction methods for finite dimensional quadratic Lie algebras. The first one gives an inductive classification of quadratic Lie algebras in characteristic zero, while the latest one produces quadratic non-associative algebras of dimension even (not only Lie) out of arbitrary ones in characteristic different from $2$. In the case of nilpotent Lie algebras, a classification scheme based on free nilpotent algebras and their invariant forms was introduced in \cite{benito2017free}.

The double extension process appears during the 1980s in several independent works of Keith and Hofmann (see~\cite{keith1984invariant}), Favre and Santharoubane and Medina and Revoy (see~\cite{favre_santharoubane_1987,medina1985algebres}). According to \cite{favre_santharoubane_1987}, this procedure follows the main ideas that V. G. Kac had written in several exercises for his students (see~\cite[Exercise 2.10 and 2.11]{kac1990infinite}). In Exercise 2.10, the double one-dimensional extension is defined, and the fact that every indecomposable solvable quadratic algebra of dimension $n+2$ can be obtained from a quadratic algebra of dimension $n$ is established in Exercise 2.11 (see \cite[Proposition 2.9]{favre_santharoubane_1987} for a complete proof). In 1997, Bordemann introduced the $T^*$-extension technique that can be applied to all known classes of non-associative algebras over fields of characteristic different from $2$. The method produces quadratic algebras of dimension even and Witt index a half of its dimension. From \cite[Theorem 3.2]{bordemann1997nondegenerate}, $T^*$-extensions are just quadratic non-associative algebras of dimension $2n$ that contain an isotropic subspace $U$ of dimension $n$ such that $U^2=0$. 
Both methods are good to produce examples, but they present difficulties when dealing with the classification problem. In the 2000s, the notion of a \emph{quadratic extension} was introduced as an attempt to understand general pseudo-Riemannian symmetric spaces (see \cite{kath2006metric} and \cite{kath2009structure}). Quadratic extensions are close relatives of double extensions.

A Lie algebra $A$ such that is centre is contained in its derived algebra is said \emph{reduced}. The pair $(r,s) = (\dim A^2, \dim Z(a))$ is called the \emph{type of the algebra}.

If the Lie algebra is quadratic, by orthogonality, its dimension is just $r+s$. In 1962, S. T. Tsou (see~\cite{tsou1962xi}) established an existence theorem for real quadratic Lie algebras of arbitrary type. The proof of this result, that was announced in \cite[Section 7]{Tsou_Walker_1957}), involves structure constants, trivectors and solutions of nonlinear systems of equations, so multilinear algebra. These ideas are in the base of the proposed scheme, as well as results on the classification of quadratic nilpotent Lie algebras given in \cite{benito2017free}. In the case of quadratic nilpotent algebras of nilindex $2$ ($2$-step nilpotent in the sequel), elemental multilinear techniques lead in \cite{benito2019quadratic} to the notion of \emph{$n$-quadratic family of matrices} and produce a computational algorithm to build this type of algebras.

In this paper, we will focus on multilinear tools to provide a constructive equivalence theorem (Theorem \ref{thm:equivalence}) that relates multi-step one-dimensional double extensions, $T^*$-extensions and $n$-quadratic families. This equivalence reduces, in a direct and natural way, the classification up to isometries of quadratic $2$-step nilpotent Lie algebras to that of $3$-alternating forms up to equivalence (Theorem \ref{thm:trivector-classification}) following the ideas of \cite{noui1997algebres}. Our equivalence also shows that invariant forms of the subclass of reduced quadratic $2$-step are hyperbolic. Even more, the class of quadratic $2$-step Lie algebras agree with the class of $T^*$-extensions of abelian Lie algebras given by non-degenerate (equivalently linearly surjective) $2$-cocycles (Corollary \ref{cor:reduced}). This assertion collects and expands Proposition 11 in \cite{Duong_2013}.

The paper splits into four sections including this introduction. In Section 2, we introduce main terminology on quadratic Lie algebras. Subsections \ref{sub:double}, \ref{sub:Tstar} and \ref{sub:quadratic_families} describe the construction methods, and give examples and patterns. Starting with the trivial quadratic algebra $(A_0=0, \varphi_0=0)$, Subsection \ref{sub:double} includes, as its main tool, a general multi-step hyperbolic extension inductively built by one-dimensional double extensions. Section 3 is devoted to our main result, Theorem \ref{thm:equivalence}. This result gives us $2$-step quadratic Lie algebras by structure constants encoded in cyclic $2$-cocycles and provides a simple rule for switching from one construction method to another. In the final section, the bijective map described in \eqref{eq:biy_coc_tri} yields to the explicit relationship between cyclic $2$-cocycles and trivectors. This bijection leads to the bijection up to isomorphisms of reduced quadratic $2n$-dimensional $2$-step Lie algebras and $n$-rank trivectors up to equivalence. As an easy application, we list the $22$ non-isometrically isomorphic reduced quadratic $2$-step Lie algebras up to dimension $16$ over the complex field.


\section{Generalities and methods}\label{sec:methods}

Along this paper, all vector spaces are considered finite-dimensional over a field $\ku$ of characteristic zero. Although, it is worth mentioning many of these results can be established in characteristic different from two.

In general, $A$ will denote a \emph{Lie algebra} over $\mathbb{K}$ and $\varphi\colon A\times A\to \mathbb{K}$ a \emph{bilinear form}. So, the bracket product $[x,y]$ of $A$ is skew-symmetric and satisfies the \emph{Jacobi identity}, i.e. $[x,[y,z]]+[y,[z,x]]+[z,[x,y]]=0$. From Jacobi and skew symmetry, the left multiplication by $x\in A$, $\ad x=[x,\blank]$, is a derivation of $A$ known as \emph{inner derivation}. When working with Lie algebras, condition (\ref{eq:invariante}) of invariance of $\varphi$ can be rewritten as
\begin{equation}\label{eq:lie_invariant}
	\varphi([x,y],z)+\varphi(y,[x,z]) = 0\quad \forall x,y,z\in A,
\end{equation}
and the inner derivation $\ad x=[x,\blank]$ is said $\varphi$-skew symmetric. In general, a derivation $d$ of $A$ such that $\varphi(d(x),y)+\varphi(x,d(y)) = 0$ will be called \emph{$\varphi$-skew symmetric} and $\der_\varphi A$ will denote the set of skew-symmetric derivations with respect to $\varphi$, while $\der A$ ($\inner A$) will denote the whole set of derivations (inner derivations) of $A$. In this way, we arrive to the notion of quadratic Lie algebra:

\begin{definition} 
	A quadratic Lie algebra $(A, \varphi)$ is a pair formed by a Lie algebra $A$ equipped with a non-degenerate invariant, see condition~\eqref{eq:lie_invariant}, symmetric bilinear form $\varphi\colon A \times A \to \mathbb{K}$.
\end{definition}

For arbitrary subsets $S$ and $T$ of $A$, $[S,T]$ denotes the $\mathbb{K}$-linear span
$[S,T]=\spa \langle [s,t]: s\in S, t\in T\rangle$. The Lie algebra $A$ is said to be $t$-step nilpotent if $A^{t+1}=0$ when $A^{t+1} = [A,A^{t}] = 0$, but $A^{t} \neq 0$, starting  $A^1=A$. We also say that $t$ is the index of nilpotency of $A$ or \emph{nilindex}. The chain of ideals
\begin{equation*}
	A\supseteq A^2 \supseteq \dots \supseteq A^t\supseteq A^{t+1}\supseteq \cdots
\end{equation*}
is the well-known \emph{lower central series} of $A$. Hence, if $A$ is $t$-step nilpotent, this series arrives to $0$ in $t+1$ steps. The \emph{upper central series} of $A$ is inductively defined by $Z_1(A)=Z(A)$ and $Z_{t+1}(A)=\{x\in A: [x,A]\subseteq Z_t(A)\}$. For every quadratic ($A$, $\varphi$), the previous series are related between them by the orthogonal condition $(A^k)^\perp=Z_{k-1}(A)$ (see \cite[Proposition~2.1]{bordemann1997nondegenerate} or \cite{medina1985algebres}). In particular,
\begin{equation}\label{eq:centre_orto_der}
	(A^2)^\perp=Z(A)\quad \mathrm{and\ therefore}\quad \dim A=\dim A^2+\dim Z(A).
\end{equation}

Among all quadratic Lie algebras, we are going to focus on the reduced ones:

\begin{definition}
	A Lie  algebra $A$ is said to be reduced if its centre is contained in its square, i.e. $Z(A)\subseteq A^2$. In the 2-step case, this is equivalent to $Z(A) = A^2$ as the other inclusion comes from being nilpotent.
\end{definition}

\noindent The reason why we limit our classification to this type of algebras is explained in~\cite[Theorem 6.2]{Tsou_Walker_1957} which says:

\begin{theorem}\label{thm:decomposition}
	Any non-reduced and non-abelian quadratic Lie algebra $(A,\varphi)$ decomposes as an orthogonal direct sum of proper ideals, $A = \fn \oplus \fa$, where $\varphi = \varphi_1 \perp \varphi_2$ and $(\fn, \varphi_1)$ is a quadratic reduced Lie algebra and $(\fa,\varphi_2)$ is a quadratic abelian algebra. \hfill $\square$
\end{theorem}

A quadratic algebra $(A,\varphi)$ is called \emph{decomposable} if it contains a proper ideal $I$ that is non-degenerated (i.e. $\varphi\mid_{I\times I}$ is non-degenerate), and \emph{indecomposable} otherwise. So, every abelian algebra of dimension greater than one or non-reduced algebra is decomposable. Any quadratic Lie algebra is the orthogonal direct sum of indecomposable quadratic Lie algebras. This assertion follows easily from the fact that $I$ is an ideal of $A$ if and only if its orthogonal space $I^\perp$ also is.

Finally, we will denote as $\fn_{d,t}$ the free $t$-step nilpotent Lie algebra on $d$ generators (see~\cite{benito2017free} for a formal definition).

There exist several ways to construct quadratic Lie algebras or equivalent structures. We give an overview of three of them, with focus on the 2-step case. We are going to see all methods in the order they first appeared.

From now on, $(A,f)$ will denote a finite-dimensional quadratic Lie algebra, while $A^*$ will be the  dual space of $A$. Moreover, $\ad^*$ will represent the coadjoint representation, so for any $\alpha\colon A\to \ku$, $\alpha \in A^*$ and $a,a' \in A$:
\begin{equation}\label{eq:coadjoint}
	\ad^*(a)(\alpha)(a') = -\alpha([a,a'])=-\alpha\circ \ad a(a').
\end{equation}


\subsection{Double extension}\label{sub:double}

Chronologically, this is the first classical method to construct all quadratic Lie algebras (see~\cite{figueroa1996structure} for a nice presentation). This is an iterative process, simultaneously introduced in the early 1980s by several authors, that allows us to find new quadratic Lie algebras starting from a smaller dimensional one.
The formal description we present here follows from \cite[Theorem 2.2]{bordemann1997nondegenerate}) which explicitly describes a similar result mentioned in \cite[Section~2.2]{medina1985algebres}.

\begin{theorem}\label{thm:doubleExtension}
	Let $(A, f)$ be a finite-dimensional quadratic Lie algebra over a field $\mathbb{K}$. Let $B$ be another finite-dimensional Lie algebra over $\mathbb{K}$ and suppose there is a Lie homomorphism $\phi\colon B \to \der_f(A)$ from $B$ onto the space of all $f$-skew-symmetric derivations of $A$. Denote by $w\colon A\times A \to B^*$ the bilinear skew-symmetric map $(a, a') \mapsto (b \mapsto f(\phi(b)(a),a'))$. Take the vector space direct sum $A_B:= B \oplus A \oplus B^*$ and define the following multiplication for $b$, $b' \in B$, $a$, $a'\in A$, and $\beta, \beta\in B^*$:
	\begin{multline}\label{eq:de_bracket}
		[b + a + \beta, b' + a' + \beta'] :=  [b,b']_B + \phi(b)(a') - \phi(b')(a) + [a,a']_A \\+ w(a, a') + \ad^*(b)(\beta') - \ad^*(b')(\beta).
	\end{multline}
	Moreover, define the following symmetric bilinear form $f_B$ on $A_B$:
	\begin{equation}\label{eq:de_bilinear}
		f_B(b + a + \beta, b' + a' + \beta') := \beta(b') + \beta'(b) + f(a, a').
	\end{equation}
	Then the pair $(A_B, f_B)$ is a quadratic Lie algebra over $\mathbb{K}$ and is called the double extension of $A$ by $(B, \phi)$. \hfill $\square$
\end{theorem}

So, the double extension, as name suggests, consist of two extensions: it is the semidirect product of two Lie algebras, $B$ and $(A\oplus_w B^*)$, where this last algebra comes from a central extension of algebra $A$. We also note that when double extending a nilpotent Lie algebra, in case the result is nilpotent, we always keep or increase its nilpotency index as next lemma proves.

\begin{lemma}\label{lem:nilpotencyDoubleExtension}
	Let $(A_B,f_B)$ be the double extension of $(A,f)$ by $(B,\phi)$. If $A_B$ is $t$-step nilpotent then $A$ is $n$-step with $n\leq t$.
\end{lemma}
\begin{proof}
	The proof is straightforward. As every $a\in A$ can be seen as even an element of $A_B$, and $[a,a']_{A_B}=[a,a']_A+w(a,a')$. If $[a,a']_{A_B} = 0$, then $[a,a']_A=0$, and $w(a,a')=0$ because the first part lays on $A$, while the second belongs to $B^*$.
\end{proof}
This result cannot be improved as an abelian quadratic Lie algebra can generate a $n$-step nilpotent one, for every $n \in \mathbb{N}$. To see it we have the following example.
\begin{example}
	Let us take the abelian quadratic Lie algebra $(A,f)$ of dimension $2n$ with basis $\{e_{-n}, e_{-n+1}, \ldots, e_{-1},e_1,\ldots,e_n\}$, where
	\begin{equation*}
		f(e_i,e_j) =
		\begin{cases}
			1 & \text{if } |i-j| = n+1,    \\
			0 & \text{if } |i-j| \neq n+1.
		\end{cases}
	\end{equation*}
	Let us consider the $f$-skew-symmetric derivation $d\colon A \to A$ where $d(e_{-i}) = e_{-i+1}$, $d(e_i) = -e_{i-1}$ for $i=2,\ldots,n$ and $d(e_1) = d(e_{-1}) = 0$. Now, we can build the double extension $(A_B, f_B)$ of $(A,f)$ by $(\phi, B)$ where $B = \mathbb{K}b$ and $\phi(b) = d$. This new algebra satisfies
	\begin{equation*}
		(A_B)^t = \spa\langle e_{-n+t-1},\ldots, e_{-1},e_1,\ldots,e_{n-t+1}\rangle.
	\end{equation*}
	Thus, $(A_B)^n$ is the bidimensional ideal linearly generated by $\{e_{-1},e_1\}$ and $(A_B)^{n+1} = 0$. Therefore, $A_B$ is $n$-step nilpotent. \hfill $\square$
\end{example}

Every non-abelian solvable (nilpotent) quadratic Lie algebra $(S,q)$ has a nonzero element $z\in S^2\cap Z(S)$, so $f(z,z)=0$ and
\begin{equation*}
	(A,f)=\left(\frac{(\ku z)^\perp}{\ku z}, f\big|_{\frac{(\ku z)^\perp}{\ku z}}\right)
\end{equation*}
is a solvable (nilpotent) quadratic algebra of dimension $\dim S-2$. According to \cite[Proposition 2.9]{favre_santharoubane_1987}, $(S,q)$ is isometrically isomorphic to a double extension of the algebra $(A,f)$ by a one-dimensional algebra $B=\ku b$. Iterating this one-dimensional process, we get that the class of solvable (nilpotent) quadratic Lie algebras is a direct sum of abelian and one-dimensional double extensions of solvable (nilpotent) ones \cite[Th\'eor\`eme III]{medina1985algebres}. So, although Theorem~\ref{thm:doubleExtension} gives us the general double extension method, for our goal, a subset of these extensions is enough. This leads us to the following definition.

\begin{definition}
	We call $(A_b, f_b)$ the one-dimensional double extension of $(A,f)$ by $(b,d)$ to the double extension $(A_B, f_B)$ of $(A,f)$ by $(B, \phi)$ where $B=\mathbb{K}b$ has dimension~1 and $\phi(b)=d$.
\end{definition}

Now, when considering the vector space $A_b=\ku b\oplus A \oplus \ku \beta$, where $\beta(b)=1$ (dual $1$-form of $b$), the Lie bracket from equation~\eqref{eq:de_bracket} turns into

\begin{equation*}
	[b_ib + a_i + \beta_i\beta, b_jb + a_j + \beta_j\beta] :=  b_id(a_j) - b_jd(a_i) + [a_i,a_j] + f(d(a_i), a_j)\beta,
\end{equation*}
for every scalar $b_i, b_j, \beta_i, b_j\in\mathbb{K}$ and $a_i\in A$.
While the symmetric bilinear form from equation~\eqref{eq:de_bilinear} can be written as
\begin{equation*}
	f_b(b_ib + a_i + \beta_i\beta, b_jb + a_j + \beta_j\beta) := b_i\beta_j + b_j\beta_i + f(a_i, a_j).
\end{equation*}

On the other hand, the derived algebra can be easily computed as
\begin{equation}\label{eq:ab2}
	A_b^2 = \im d + \spa\langle [a_1,a_2]_A + f(d(a_1),a_2)\beta : a_1,a_2 \in A\rangle.
\end{equation}

The description of the centre and, therefore, the reducibility of $A_b$ depends on if $d$ is either an inner or an outer $f$-skew-symmetric derivation of $A$.

\begin{lemma}\label{lem:centre}
	Let $(A_b,f_b)$ be the one-dimensional double extension of $(A,f)$ by $(b,d)$. Then
	\begin{equation*}
		Z(A_b) = (Z(A)\cap\ker d) \oplus B^*
	\end{equation*}
	if and only if $d\notin \inner A$. Otherwise, $d=\ad x$ for some $x\in A$ and
	\begin{equation*}
		Z(A_b) = (Z(A)\cap\ker d) \oplus B^* \oplus \ku (b-x)
	\end{equation*}

\end{lemma}
\begin{proof}

	If we calculate the centre we obtain
	\begin{align*}
		Z(A_b) & = \{b_1b+a_1+\beta_1\beta : b_1d(a_2) - b_2d(a_1)+[a_1,a_2]=0,                                                            \\
		       & \hspace{5.5cm} f(d(a_1),a_2) = 0 \ \forall a_2\in A, b_2 \in \mathbb{K}\}                                                 \\
		       & = \{b_1b+a_1+\beta_1\beta : b_1d(a_2) + [a_1,a_2]=0, d(a_1) = 0 \ \forall a_2\in A\}                                      \\
		       & = \{a_1+\beta_1\beta : [a_1,a_2]=0, d(a_1) = 0 \ \forall a_2\in A\}                                                       \\
		       & \hspace{3.6cm}+ \spa\left\langle b+\frac{1}{b_1}\,a_1: d = -\ad\!\left(\frac{1}{b_1}\,a_1\right), b_1 \neq 0\right\rangle \\
		       & = \begin{cases}
			(Z(A)\cap \ker d) \oplus B^*                 & \text{ if } d\notin\inner A, \\
			(Z(A)\cap \ker d) \oplus B^* \oplus \ku(b-x) & \text{ if } d = \ad x.
		\end{cases}
	\end{align*}

	Note that $d=\ad x =\ad y$ if and only if $x-y \in Z(A)\cap \ker d$.
\end{proof}

\begin{corollary}
	Let $(A_b,f_b)$ be the one-dimensional double extension of $(A,f)$ by $(b,d)$. Then $(A_b,f_b)$ is reduced if and only if $d\notin \inner A$, $A_b^2=(\im d +A^2)\oplus B^*$ and $Z(A)\cap\ker d \subseteq \im d +A^2$.
\end{corollary}
\begin{proof}
	Since $A_b$ is reduced if and only if $Z(A_b) \subseteq A_b^2$, the result is straight forward applying Lemma~\ref{lem:centre} and equation~\eqref{eq:ab2}.
\end{proof}

Now, as the aim of this paper is to study the 2-step case, we will see which restrictions do $A$ and $d$ need to satisfy in the following proposition.

\begin{proposition}\label{prop:2step}
	Let $(A_b,f_b)$ be the one-dimensional double extension of a quadratic Lie algebra $(A,f)$ by $(b,d)$. Then  $A_b$ is 2-step if and only if
	\begin{equation*}
		0 \neq \im d + A^2\subseteq Z(A) \cap \ker d.
	\end{equation*}
\end{proposition}
\begin{proof}
	First, we have that

	\begin{multline*}
		A_b^3 = d^2(A) + d(A^2) + \spa\langle [d(a_1),a_2] + f(d^2(a_1),a_2)\beta : a_i \in A\rangle\\
		+ \spa\langle f(d([a_1,a_2]),a_3)\beta : a_i \in A\}.
	\end{multline*}
	As this must be zero, we need

	\begin{equation}\label{eq:conditions2step}
		\begin{cases}
			d(A)\subseteq Z(A), \\
			d(d(A))=d^2(A)=0,   \\
			d([A,A])=d(A^2)=0.
		\end{cases}
	\end{equation}
	The conditions in equation~\eqref{eq:conditions2step}, as $d$ is a derivation, can be expressed in one line as
	\begin{equation}\label{eq:conditions2stepSimplified}
		\im d + A^2\subseteq Z(A) \cap \ker d.
	\end{equation}
	At this point, $A_b^3 = 0$, and we need to check if $A_b^2 \neq 0$ in case $A$ is abelian. This, using equation~(\ref{eq:ab2}), translates into

	$d\neq 0$ finishing the proof.
\end{proof}
\begin{remark}
	Note that every homomorphism $d\colon A\to A$ that satisfies condition~\eqref{eq:conditions2stepSimplified} is indeed a derivation as $\im d\subseteq Z(A)$ and $A^2\subseteq \ker d$.
\end{remark}
\begin{corollary}
	Let $(A_b,f_b)$ be the one-dimensional double extension of an abelian quadratic Lie algebra $(A,f)$ by $(b,d)$. Then $A_b$ is 2-step if and only if $d\neq 0$ and $d^2=0$.
\end{corollary}

As we have previously noted, solvable Lie algebras can be obtained by iterating one-dimensional double extensions. This multistep procedure can be implemented in a nested way, which is the idea in our following construction.

\subsection*{Chained one-dimensional double extensions construction}
We are going to consider a chain of one-dimensional double extensions $\{(A_k,f_k)\}_{k=0}^n$. We start by introducing the following notation for every algebra of our chain:

\begin{align*}
	A_{k+1} = B_{k+1} \oplus A_k \oplus B^*_{k+1},
\end{align*}
where $B_k = \mathbb{K}b_k$ and $B^*_k = \mathbb{K}b_k^*$ are 1-dimensional and $\dim A_k = 2k$.
Now, we can also define $A_{k+1} = A_{k+1,1} \oplus A_{k+1,2}$ where
\begin{align*}
	A_{k+1,1} & = B_{k+1} \oplus A_{k,1},   \\
	A_{k+1,2} & = A_{k,2} \oplus B_{k+1}^*.
\end{align*}
Applying this definition recursively we obtain
\begin{equation}\label{eq:decomposition_Ak}
	A_{k+1,1} = \bigoplus_{i=1}^{k+1} B_i, \qquad A_{k+1,2} = \bigoplus_{i=1}^{k+1} B_i^*.
\end{equation}
All this algebras $A_k$ are associated to an invariant bilinear form $f_k$. Moreover, over them we define derivations $d_k\colon A_k \to A_k$ such that $d_k \in \der_{f_k}(A_k)$ to do the double extensions. Hence, $(A_{k+1}, f_{k+1})$ is the one-dimensional double extension of $(A_k,f_k)$ by $(b_{k+1},d_k)$, starting with $A_0 = \{0\}$ and $f_0=0$. Note this is a really convenient notation. First, it gives us a basis for $A_{k}$:
\begin{equation*}
	\{b_k, b_{k-1}, \ldots, b_1, b_1^*, b_2^*,\ldots,b_k^*\},
\end{equation*}
where the order of this basis is given by the chain itself. Even more, if we divide the set separating $b_i$ from $b^*_j$ elements, we get the bases for $A_{k,1}$ and $A_{k,2}$ respectively. All together, we can see this build as a telescopic construction in Figure~\ref{fig:telescopic}.
\begin{figure}[hbt]
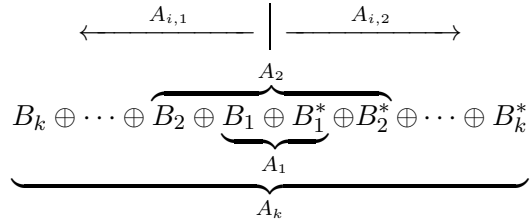

	\newlength{\arrowsep}
	\setlength{\arrowsep}{0.8cm}
	\begin{gather*}
		\begin{array}{c|c}
			\xleftarrow{\hspace{\arrowsep}A_{i,1}\hspace{\arrowsep}} & \xrightarrow{\hspace{\arrowsep}A_{i,2}\hspace{\arrowsep}}
		\end{array}\\
		\underbrace{B_{k} \oplus \cdots \oplus \overbrace{B_2 \oplus \underbrace{B_1 \oplus B_1^*}_{A_1} \oplus B_2^*}^{A_2} \oplus \cdots \oplus B_{k}^*}_{A_{k}}
	\end{gather*}
	\caption{Telescopic view of the chained one-dimensional double extension. 
	}
	\label{fig:telescopic}
\end{figure}

\noindent Let us now define for $k=0,\ldots,n-1$
\begin{align}\label{eq:wk_definition}
	\begin{split}
		w_{k+1}\colon A_{k} \times A_{k} &\to B_{k+1}^*\\
		(a,b) &\mapsto f_{k}(d_{k}(a),b)\,b_{k+1}^*.
	\end{split}
\end{align}
So, in basis
\begin{equation*}
	\{b_{k+1},b_k,\ldots,b_1,b_1^*,\ldots,b_k^*,b_{k+1}^*\},
\end{equation*}
we can give the Lie bracket $[\blank,\blank]_{k+1}$ of algebra $A_{k+1}$ which is
\begin{alignat}{2}
	[b_{k+1}^*, \blank]_{k+1} & = 0,			\qquad           & [b_i, b_j]_{k+1}     & = [b_i,b_j]_k + w_{k+1}(b_i, b_j),\nonumber     \\\label{eq:mulTable}
	[b_{k+1}, b_i]_{k+1}      & = d_k(b_i) ,  \qquad & [b_i^*, b_j^*]_{k+1} & = [b_i^*,b_j^*]_k + w_{k+1}(b_i^*, b_j^*),      \\
	[b_{k+1},b_i^*]_{k+1}     & = d_k(b_i^*),\qquad  & [b_i, b_j^*]_{k+1}   & = [b_i,b_j^*]_k + w_{k+1}(b_i, b_j^*),\nonumber
\end{alignat}
for $1 \leq i,j \leq k$. While the bilinear form satisfies
\begin{equation*}
	\begin{cases}
		f_{k+1}(b_{k+1},b_{k+1}^*) = 1,                                                  \\
		f_{k+1}(b_{k+1},B_{k+1}\oplus A_k) = f_{k+1}(b_{k+1}^*,A_k\oplus B_{k+1}^*) = 0, \\
		f_{k+1}\big|_{A_k\times A_k} = f_k.
	\end{cases}
\end{equation*}

\begin{remark}\label{rmk:chained-bilform}
	Note $A_{k+1,2}=(A_{k+1,2})^\perp$. So, $A_{k+1,2}$ is a \emph{lagrangian for $f_{k+1}$} (see~\cite[Chapter I section 1.C]{elman2008algebraic}) and therefore, $f_{k+1}$ is a \emph{metabolic} or \emph{hyperbolic} symmetric form (they are equivalent terms in characteristic different from 2).
\end{remark}

\begin{remark}\label{rmk:chained-2nil}
	From Lemma \ref{lem:nilpotencyDoubleExtension}, we have $A_{k+1}$ can be $t$-nilpotent only if $A_k$ is $n$-nilpotent with $n\leq t$. Hence, combining this result in the case $t=2$ with Proposition~\ref{prop:2step}, we conclude that $A_{k+1}$ is $2$-step if and only if
	\begin{itemize}
		\item $A_k$ is abelian and $0\neq \im d_k\subseteq \ker d_k$ or,
		\item $A_k$ is $2$-step, and $\im d_k\subseteq Z(A_k)\cap \ker d_k$ and $A_k^2 \subseteq \ker d_k$.
	\end{itemize}
\end{remark}
This remark is useful when searching for $2$-step quadratic Lie algebras, and it leads us to the following definition:
\begin{definition}\label{def:2chained}
	Let $\{(A_k,f_k)\}_{k=0}^{n}$ be a chain of algebras obtained from successive one-dimensional double extensions from the previous one in the chain by $\{b_{k+1},d_k\}_{k=0}^{n-1}$ starting from $A_0=\{0\}$ and $f_0=0$. We say the chain satisfies:
	\begin{itemize}
		\item the \emph{non-null property} (NNP) if there exists $k$ such that $d_{k}\neq 0$,
		\item and the \emph{two-step property} (2SP) if $\im d_k \subseteq A_{k,2} \subseteq \ker d_k$ for every $k\geq 1$.
	\end{itemize}
\end{definition}

In any $\{(A_k,f_k)\}_{k=0}^{n}$ chain of one-dimensional double extensions, we have $d_0=d_1=0$ and, therefore, $A_1$ and $A_2$ are abelian quadratic algebras of dimension $2$ and $4$ respectively. But for greater dimensions we can obtain 2-step algebras. If our chain satisfies (NNP) and (2SP) properties from Definition \ref{def:2chained}, we can easily check by applying (2SP) inductively
\begin{equation}\label{eq:chainContained}
	A_k^2 \subseteq A_{k,2} \subseteq Z(A_k).
\end{equation}
Therefore, every step or link of this chain satisfies equation~\eqref{eq:conditions2stepSimplified}. Hence, its final quadratic Lie algebra $(A_n,f_n)$ for $n\geq 3$ is 2-step by using the (NNP) as observed in Remark~\ref{rmk:chained-2nil}. In addition, we have chosen these properties as we want to end up with a reduced 2-step Lie algebra $A_n$ whose square or centre is $A_{n,2}$ and, as images and kernels determine its product, this is the most natural way to obtain it. The important result we are going to prove later in Theorem~\ref{thm:equivalence} is that we can obtain all reduced quadratic 2-step Lie algebras as the final quadratic Lie algebra $(A_n,f_n)$ of some chain of one-dimensional double extensions that satisfies (NNP) and (2SP).

\begin{lemma}\label{lem:chain}
	Let $\{(A_k,f_k)\}_{k=0}^{n}$ be a chain of one-dimensional double extensions by $\{(b_{k+1},d_k)\}_{k=0}^{n-1}$ satisfying \emph{(NNP)} and \emph{(2SP)}. And let define
	\begin{equation*}
		D_{ijk} := \sgn(\sigma)f_{\sigma(k)-1}(d_{\sigma(k)-1}(b_{\sigma(i)}),b_{\sigma(j)})
	\end{equation*}
	for some permutation $\sigma$ such that $1\leq \sigma(i)<\sigma(j)<\sigma(k)$ or $D_{ijk} = 0$ if some subindexes repeat.  Then $n\geq 3$ and
	\begin{enumerate}[\qquad (a)]
		\item $(A_n,f_n)$ is a $2n$-dimensional $2$-step quadratic Lie algebra 
		      such that $A_{n,2}
			      \subseteq Z(A_n)$, $[b_i, b_j]_n = \sum_{k=1}^{n} D_{ijk}\,b_k^*$, and the invariant bilinear form $f_n$ is given by $f_n(b_i, b_j^*) = \delta_{ij}$ and $f_n(b_i, b_j) = f_n(b_i^*, b_j^*) = 0$. 
		\item $(A_n,f_n)$ is reduced if and only if
		      \begin{equation*}
			      A_{n,2}=\spa \left\langle\; \sum_{k=1}^n \hat{w}_k(b_i,b_j): 1\leq j<i\leq n\right\rangle
		      \end{equation*}
		      where $\hat{w}_k$ is the alternating extension of $w_k$ defined in equation~\eqref{eq:wk_definition}.
	\end{enumerate}

\end{lemma}

\begin{proof}
	First of all, we can observe $D_{ijk}$ definition resembles the idea $d_{k-1}$ is $f_{k-1}$-skew-symmetric because $D_{ijk}=-D_{ijk}$ and $D_{iik}=0$ as $f_{k-1}(d_k(b_i),b_j)$ for $i,j \leq k$ in characteristic different from 2.

	Next, from previous arguments after Definition \ref{def:2chained}, we have that $(A_n,f_n)$ is a quadratic $2$-step Lie algebra. Now, applying multiplication table in equation~\eqref{eq:mulTable} recursively and using $A_{k,2} \subseteq \ker d_k$ by (2SP), we obtain
	\begin{equation}\label{eq:productChainLong}
		\begin{cases}
			[b_i, b_j]_{k+1} = w_{k+1}(b_i,b_j) + w_{k}(b_i,b_j) + \ldots + w_{i+1}(b_i,b_j) + d_{i-1}(b_j), \\
			[b_i^*, \cdot\,]_{k+1} = 0,
		\end{cases}
	\end{equation}
	for $ 1 \leq j<i\leq k+1$. We also have for $ 1 \leq i,j \leq k+1$
	\begin{equation*}
		\begin{cases}
			f_{k+1}(b_i, b_j^*) = \delta_{ij},             \\
			f_{k+1}(b_i, b_j) = f_{k+1}(b_i^*, b_j^*) = 0.
		\end{cases}
	\end{equation*}
	Moreover, (2SP) also implies that $d_{i-1}(A_{i-1})\subseteq A_{{i-1},2}=\spa \langle b_1^*,\dots, b_{i-1}^*\rangle$ and for $j<i$,
	\begin{equation*}
		d_{i-1}(b_j) = \sum_{k=1}^{i-1}f_{i-1}(d_{i-1}(b_j),b_k)b_k^* = \sum_{k=1}^{i-1}D _{jki}\,b_k^* = \sum_{k=1}^{i-1}D_{ijk}\,b_k^*,
	\end{equation*}
	by using $D_{ijk}$ definition in the last equality. So product~\eqref{eq:productChainLong} in $(A_n,f_n)$ turns into
	\begin{equation*}
		\begin{cases}
			[b_i, b_j]_n = \sum_{k=1}^{n} D_{ijk}\,b_k^*, \\
			[b_i^*, \cdot\,]_n = 0.
		\end{cases}
	\end{equation*}
	And even more,
	\begin{equation*}
		f_n([b_i,b_j]_n, b_k) = D_{ijk}.
	\end{equation*}
	Now $A_n$ is a $2$-step nilpotent Lie algebra, thus being reduced is equivalent to $Z(A_n)=A_n^2$.
	Now if we define $\hat{w}_k(b_i,b_k) = w_k(b_i,b_j)$

	when $i,j<k$ and $\hat{w}_k(b_i,b_j) = \sgn(\sigma) w_{\sigma(k)}(b_{\sigma(i)},b_{\sigma(j)})$

	where $\sigma$ is some permutation of $\{i,j,k\}$ such that $\sigma(k) = \max\{i,j,k\}$,
	\begin{equation*}
		\ad_{A_n} b_i (b_j) = \sum_{k=0}^{n-1} \hat{w}_{k+1}(b_i,b_j).
	\end{equation*}
	Hence $A_n^2=\spa \langle \sum_{k=1}^n \hat{w}_k(b_i,b_j): 1\leq j<i\leq n\rangle$ and applying equation~\eqref{eq:centre_orto_der} for $k=n$ we finish the proof.
\end{proof}

All these relations and notation will serve us later to prove the equivalence between the different approaches for constructing these algebras.

\subsection{T*-extension}\label{sub:Tstar}

The $T^*$-extension is a one-step method, which was introduced by Bordemann in 1997. In contrast to double extension, it can be applied not only to Lie algebras, but to arbitrary non-associative algebras. Nevertheless, as we are focused on the study of Lie algebras, we will only see its definition applied on these algebras (for a general definition see \cite{bordemann1997nondegenerate}). 
Along this subsection, $(B,[x,y]_B)$ will be a Lie algebra.

Let $V$ a $B$-module given by the representation $\rho\colon B\to \mathfrak{gl}(V)$ (indeed a Lie algebra homomorphism as $\mathfrak{gl}(V)$ denotes the general Lie algebra of endomorphisms over the vector space $V$). In order to reach the definition of $T^*$-extension we need the following basic cohomology notions.

\begin{definition}
	Let $w\colon B\times B \to V$ be a bilinear map, $V$ a $B$-module given by the representation $\rho$, and $a,b,c$ arbitrary elements of $B$. Then we say:
	\begin{itemize}
		\item $w$ is \emph{non-degenerate} if its radical is zero, i.e. $\rad w=\{b\in B: w(b,\blank)=0\} = 0$;
		\item $w$ is \emph{cyclic} if
		      \begin{equation}\label{eq:cyclic}
			      w(a,b)(c) = w(c,a)(b) = w(b,c)(a);
		      \end{equation}
		\item and $w$ is a \emph{$2$-cocycle} if $w$ is skew-symmetric and
		      \begin{equation*}
			      \sum_{\overset{\circlearrowright}{a,b,c}}w([a,b],c) = \sum_{\overset{\circlearrowright}{a,b,c}} \rho(a) w(b,c).
		      \end{equation*}
	\end{itemize}
	The vector space of $2$-cocycles with values in $V$ is denoted by $Z^2(B,V)$.
\end{definition}
\begin{remark}\label{rmk:2cocyc_3cocyc}
	Given a bilinear map $w\colon B\times B\to B^*$, we can define the trilinear map $\phi_w: B\times B\times B \to \ku$. It is straightforward to infer $w$ is cyclic, $w\in Z^2(B,B^*)$ where $V=B^*$ by the coadjoint representation if and only if $\phi_w$ is a $3$-cocycle, which means it is a $3$-alternating form such that
	\begin{multline*}
		\phi_w([b_0,b_1], b_2,b_3)+\phi_w([b_1, [b_0,b_2],b_3)+\phi_w([b_1, b_2,[b_0,b_3])=\\
		\phi_w(b_0,[b_1, b_2],b_3)+\phi_w(b_0,b_2, [b_1,b_3])-\phi_w(b_0,b_1, [b_2,b_3]).
	\end{multline*}
	The vector space of scalar $3$-cocycles is denoted by $Z^3(B,\ku)$. Here, $V=\ku$ comes from the trivial representation.
\end{remark}

Consider now an arbitrary bilinear form $w\colon B\times B \to B^*$, and define the following multiplication on the vector space $ B\oplus B^*$ for $b,b' \in B$ and $\beta, \beta' \in B^*$:
\begin{equation}\label{eq:TwProduct}
	[b+\beta,\ b'+\beta'] := [b,b']_B + w(b,b') + \ad^*(b)(\beta') - \ad^*(b')(\beta),
\end{equation}
where $\ad^*$ is the coadjoint representation defined in equation~\eqref{eq:coadjoint}.
Moreover, we construct the symmetric bilinear form $q_B$ as:
\begin{equation}\label{eq:TwForm}
	q_B(b+\beta, b'+\beta') := \beta(b') + \beta'(b).
\end{equation}

\begin{proposition}\label{prop:cond2coc}
	Let $B$, $B^*$, $w$, and $q_B$ be as above. Then:
	\begin{itemize}
		\item [\emph{(a)}] The vector space $B\oplus B^*$ with the binary product given in equation~\eqref{eq:TwProduct} is a Lie algebra if and only if $B$ is a Lie algebra and $w\in Z^2(B,B^*)$.
		\item [\emph{(b)}] The hyperbolic form $q_B$ defined in equation~\eqref{eq:TwForm} is an invariant bilinear form of the Lie algebra $B\oplus B^*$ if and only if  $w$ is cyclic.
	\end{itemize}
	So, $(B\oplus B^*, q_B)$ is a quadratic Lie algebra if and only if the bilinear form $w$ is a cyclic $2$-cocycle and $(B,[x,y]_B)$ is a Lie algebra. 
\end{proposition}
\begin{proof}
	Assertion (a) follows from \cite[p. 177]{bordemann1997nondegenerate}. Here we can see the Jacobi identity is satisfied if and only if $B$ is a Lie algebra and $w$ is a $2$-cocycle. On the other hand, product in equation~\eqref{eq:TwProduct} is skew-symmetric if and only if $[x,y]_B$ and $w$ are skew, which also comes from being a 2-cocyle. Finally, to prove assertion (b) about the $T^*$ construction, we use \cite[Lemma 3.1]{bordemann1997nondegenerate}, which adds us the cyclic condition in order to be quadratic.
\end{proof}
\begin{definition}
	Let $w\colon B\times B \to B^*$ be cyclic $2$-cocycle and $B$ Lie algebra. The quadratic algebra $(B\oplus B^*, q_B)$, with product and quadratic form defined in equations~\eqref{eq:TwProduct} and \eqref{eq:TwForm} respectively, is called the $T^*$-extension of $B$ by $w$, and we denote it as $(T^*_w B, q_B)$.
\end{definition}

Finally, we have the following theorem (see~\cite[Theorem~3.2]{bordemann1997nondegenerate}) which gives us conditions about when we can build a quadratic Lie algebra using $T^*$-extensions.
\begin{theorem}[Bordemann, 1997]\label{thm:text_isometric}
	Let $(A,f)$ be a quadratic Lie algebra of finite dimension $n$ over a field $\mathbb{K}$ of characteristic not equal to two. Then $(A,f)$ will be isometric to a $T^*$-extension $(T^*_wB, q_B)$ if and only if $n$ is even and $A$ contains an isotropic ideal
	$I$ (i.e. $I \subset I^\perp$) of dimension $n/2$. In this case, as a Lie algebra, $B$ is isomorphic to the quotient $A/I$.
\end{theorem}
\begin{remark}
	As seen in the proof, any isotropic ideal of dimension $n/2$ will work. We also note that, as stated in original paper, any isotropic subspace $V$ of $A$ whose dimension is $\dim A/2$ is an ideal of $A$ if and only if it is abelian ($V^2=0$).
\end{remark}

\begin{corollary}\label{cor:lagrangian}
	The class of $T^*$-extensions is just the class of quadratic Lie algebras of dimension $2n$ with a lagrangian $n$-dimensional ideal.
\end{corollary}

\begin{example}
	$T^*$-extensions of Lie algebras by the null bilinear form, $T^*_0B$ are just split extensions of that Lie algebra $B$ by means of its coadjoint representation. And even more, any invariant bilinear form $f\colon B\times B\to \ku$ let us define another invariant quadratic form $Q_{q_B,f}$ on $(T^*_0B,q_B)$:
	\begin{equation*}
		Q_{q_B,f}(a+\alpha, b+\beta)=q_B(a+\alpha, b+\beta) + f(a,b) = \alpha(b)+\beta(a)+f(a,b).
	\end{equation*}
	The resulting quadratic Lie algebra $(T^*_0B, Q_{q_B,f})$ was introduced in \cite{hofmann1986invariant} as the \emph{inflaction of $B$} with respect to the forms $q_B$ and $f$. According to \cite[Lemma 2.9]{hofmann1986invariant}, inflactions occur prominently in the structure of quadratic mixed Lie algebras. \hfill $\square$
\end{example}

\begin{example}
	Let $\mathfrak{h}=\spa \langle x,y, z \rangle$ be the Heissenberg $3$-dimensional Lie algebra  given by the nonzero products $[x,y]=z$. Apart from the $6$-dimensional $2$-step Lie algebra $T^*_0\mathfrak{h}$, we can construct a $6$-dimensional $3$-step Lie algebra taking the cyclic 2-cocyle $w$ defined as
	\begin{equation*}
		w(v_1,v_2)(v_3)=\left|
		\begin{array}{ccc}
			\lambda_1 & \lambda_2 & \lambda_3 \\
			\beta_1   & \beta_2   & \beta_3   \\
			\gamma_1  & \gamma_2  & \gamma_3
		\end{array}
		\right|,
	\end{equation*}
	where $v_i=\lambda_ix+\beta_iy+\gamma_iz$.  \hfill $\square$
\end{example}
Once the construction is clear, we can start by seeing what does $B$ and $w$ need to satisfy in order to obtain a 2-step quadratic algebra as we have already done in the double extension. First, analogue to Lemma~\ref{lem:nilpotencyDoubleExtension}, we have the following result about the nilpotency order of the extension. It comes from~\cite[Theorem~3.1]{bordemann1997nondegenerate} but adapting indices to our situation.
\begin{proposition}
	If $B$ is a $k$-step nilpotent Lie algebra, then for every cyclic 2-cocycle $w\colon B \times B \to B^*$ the $T^*$-extension $T^*_w B$ is $n$-step nilpotent where $k\leq n \leq 2k$.
\end{proposition}

\begin{remark}
	This result cannot be improved. Indeed, in the following sections, we build 2-step quadratic Lie algebras from abelian ones (see Corollary \ref{cor:duong}).
\end{remark}

Now, let us find which is the centre and the square of these algebras. In general,
\begin{equation*}
	Z(T^*_w B) = \{b +\beta : b \in Z(B) \text{ and } w(b, b') + \beta \circ \ad b' = 0 \enspace \forall b' \in B\},
\end{equation*}
and
\begin{equation}\label{eq:tstar2}
	(T^*_w B)^2 = \spa \langle [b,b']_B + w(b,b') : b,b' \in B\rangle +
	\spa \langle \beta \circ \ad b : b \in B, \beta \in B^*\rangle,
\end{equation}

\noindent

\begin{lemma}\label{lem:centre_der}
	For any $V$ subspace of $B$, let define $V^\circ:=\{\beta \in B^*: \beta(V)=0\}$ and $V^\perp$ its orthogonal subspace in $T^*_wB$ with respect to quadratic form $q_B$. Then, $V^\perp=B\oplus V^\circ$ and $(V^\circ)^\perp= B^*\oplus V$ and:
	\begin{enumerate}[\quad(a)]
		\item $Z(B)^\circ=\spa \langle \beta \circ \ad b : b \in B, \beta \in B^*\rangle\subseteq (T^*_w B)^2$.
		\item $(B^2)^\circ=\{ \beta\in B^*: \beta\circ \ad b=0 \enspace \forall b \in B\}$.
		\item $Z(T^*_w B) \cap B= Z(B)\cap \rad w$ and $Z(T^*_w B) \cap B^*= (B^2)^\circ$.
		\item $\spa\langle w(b,b'): b,b' \in B\rangle\subseteq (\rad w)^\circ$.
	\end{enumerate}
\end{lemma}
\begin{proof}
	Let $X=\spa \langle \beta \circ \ad b : b \in B, \beta \in B^*\rangle$ and $x\in B$ such that $0=q_B(x, \beta \circ \ad b)=\beta ([b,x]_B)$, $\forall \beta \in B^*, \forall b\in B$. Previous equality is equivalent to $[b,x]_B=0$ $\forall b\in B$, i.e., $x\in Z(B)$. Hence $Z(B)=X^\perp\cap B$ and item (a) follows from $Z(B)^\perp=X \oplus B$, $X \subseteq Z(B)^{\circ}$ and equation~\eqref{eq:tstar2}. Now, from equation~\eqref{eq:centre_orto_der}, $(Z(T^*_w B) \cap B^*)^\perp=(T^*_w B)^2 +B^*=B^2\oplus B^*=((B^2)^\circ)^\perp$ which implies item (b) and second assertion in item (c). Note, the other equality in item (c) is straightforward. Finally, if $\beta \in (\rad w)^\perp \cap B^*$, then $\beta(a)=0$ $\forall a\in B$ such that $w(a,b)=0, \forall b\in B$. Since $w$ is cyclic, for a fixed $a\in B$, $w(a,b)(b')=w(b,b')(a)$ and we get (d) when $a \in \rad w$.
\end{proof}

From Lemma \ref{lem:centre_der} we get immediately that $Z(T^*_0B)=Z(B)\oplus (B^2)^\circ$ and $(T^*_0B)^2=B^2\oplus Z(B)^\circ$. So, $T^*_0B$ is reduced $2$-step if and only if $B$ is reduced $2$-step. Even more, this lemma also shows that we can build quadratics $2$-step Lie algebras from abelian ones in an easy way:

\begin{corollary}\label{cor:reduced}
	Let $B$ be an abelian Lie algebra. Then $(T^*_w B)^2=\spa \langle w(b,b') : b,b' \in B\rangle$ and $Z(T^*_w B)=\rad w \oplus B^*$. So, $T^*_w B$ is $2$-step if and only if $w$ is not null. Moreover, it is equivalent:
	\begin{enumerate}[\quad (a)]
		\item $(T^*_w B,q_B)$ is reduced, 
		\item $w$ is non-degenerate, 
		\item $B^*=\spa \langle w(b,b'): b,b'\in B\rangle$. 
	\end{enumerate}
\end{corollary}
\begin{proof}
	The first part follows easily from $B$ being abelian, item (d) of Lemma \ref{lem:centre_der}, and the general description of $(T^*_0B)^2$. Now, $(T^*_w B,q_B)$ is reduced if and only if $(T^*_w B)^2=B^*=Z(T^*_w B)$. Hence, items (a) and (c) are equivalent. Finally, $Z(T^*_w B)=B^*$ if $w$ is non-degenerate, and then $(T^*B,q_B)^2=Z(T^*_w B)^\perp=B^*$ by using equation~\eqref{eq:centre_orto_der}.
\end{proof}

\begin{corollary}\label{cor:duong}
	Let $(A,f)$ a quadratic Lie algebra. Then, $A$ is 2-step reduced nilpotent if and only if $A$ is isometrically isomorphic to a $T^*_w B$ extension of an abelian Lie algebra $B$ where $w$ is non-degenerate.
\end{corollary}
\begin{proof}
	If $A$ is $2$-step and reduced, $Z(A)=A^2=Z(A)^\perp$ is a lagrangian ideal and, from Theorem \ref{thm:text_isometric}, algebra $A$ is isometrically isomorphic to $T^*_w B$ and $B\cong A/A^2$, so $B$ is abelian. The converse follows from Corollary \ref{cor:reduced}.
\end{proof}

\begin{remark}
	Corollary \ref{cor:duong} provides an alternatively proof of Proposition 11 in \cite{Duong_2013}. It also shows the condition of $w$ being non-degenerate can be changed by that of the dual space $B^*$ being the linear span of the image of $w$.
\end{remark}


\subsection{Quadratic Lie algebras and $\bm{n}$-quadratic families}\label{sub:quadratic_families}
According to \cite{benito2019quadratic}, the classification of quadratic nilpotent Lie algebras can be reduced, in some categorical way, to the study of symmetric invariant bilinear forms on free nilpotent Lie algebras. As stated in~\cite[Propositions 1.4 and 1.5]{gauger_1973}, and denoting by $\fn_{d,t}$ the $t$-nilpotent Lie algebra on $d$ generators, any $t$-step nilpotent Lie algebra on $d$ generators ($d=\cdim A^2$, type of the algebra in the sequel) is a homomorphic image of $\fn_{d,t}/I$ with $I$ ideal such that $\fn^t_{d,2}\subsetneq I \subsetneq \fn^2_{d,2}$. This special relation allows us to build quadratic nilpotent Lie algebras by means of free nilpotent \cite[Proposition 4.1]{benito2017free}. The main idea to do this is summarized in the following result.

\begin{proposition}[Benito et al., 2017]
	Let $A=\fn_{d,t}/I$, with $I$ ideal such that $\fn^t_{d,2}\subsetneq I \subsetneq \fn^2_{d,2}$. Then, there exists a symmetric, invariant and non-degenerate bilinear form $f\colon A\times A \to \ku$ if and only if there exists a symmetric invariant bilinear form $\psi$ on $\fn^2_{d,2}$ such that $I=\fn_{d,t}^\perp$. The relation between $f$ and $\psi$ is given by $\psi (a,b)=f(a+I,b+I)$.
\end{proposition}

For $2$-step quadratic Lie algebras, this classification process can be reformulated by using the notion of $n$-quadratic family. This is a special set of skew symmetric matrices which encodes the structural constants of a quadratic $2$-step Lie algebra of type $n$ and dimension $2n$. This approach follows preliminary ideas and techniques suggested in \cite{Tsou_Walker_1957} and relates the classification of $2$-step quadratic to that of $3$-forms (see~\cite{noui1997algebres}).

\begin{definition}
	For any $n\geq 2$, a family $\{M_1, \ldots, M_n\}$ of matrices of order $n\times n$ with entries in $\mathbb{K}$ is called $n$-quadratic if the following properties are satisfied:
	\begin{enumerate}
		\item Every matrix $M_i$ is skew-symmetric.
		\item The $i$-th column of every $M_i$ is null.
		\item For any $j>i$, the $j$-th column of $M_i$ is the additive inverse of the $i$-th column of $M_j$.
	\end{enumerate}
	In case, matrix
	\begin{equation*}
		\mathfrak{F}(M_1, \ldots, M_n) = [M_{1<j}M_{2<j}\ldots M_{d-1<j}],
	\end{equation*}
	of order $n \times \frac{n(n-1)}{2}$ has rank $n$ (maximum) we say this is a \emph{non-degenerate $n$-quadratic family}. Here $M_{i<j}$ denotes the submatrix of $M_i$ given by the set of all $j$-th columns of $M_i$ such that $i<j$.
\end{definition}

Let $(A,\varphi)$ be a quadratic 2-step Lie algebra. Since $A^2\subseteq Z(A)$, $A$ is reduced if and only if $Z(A)=A^2$, and, from equation~\eqref{eq:centre_orto_der}, $\dim A=2n$ and $d=\cdim A^2=\dim A^2$. Otherwise, Theorem \ref{thm:decomposition} tells us $A$ decomposes as an orthogonal sum of ideals $\fn \oplus \fa$ where $\fn$ is $2$-step reduced (so even dimensional) and $\fa$ is abelian. We assume in the sequel $(A,\varphi)$ is a quadratic 2-step Lie algebra of dimension $2n$.

If even $2$-step, as stated in \cite[Theorem 8]{benito2019quadratic} for reduced (see also Corollary \ref{cor:lagrangian}, since $A^2$ is lagrangian), we can find a basis, $\{v_1, \ldots, v_n, z_1, \ldots, z_n\}$, where the Lie bracket~is
\begin{align*}
	[v_i, v_j]    & := \sum_{k=1}^n m_{ijk}\,z_k, \\
	[z_i, \blank] & := 0.
\end{align*}
While the bilinear form satisfies
\begin{equation*}
	\varphi(v_i,v_j) = 0,\qquad
	\varphi(z_i,z_j) = 0,\qquad
	\varphi(v_i,z_j) = \delta_{ij}.
\end{equation*}
This means its structure constants are determined by a (non-degenerate) family of $n$-quadratic matrices $\{M_i : 1 \leq i \leq n\}$. Here, $m_{ijk}$ is the entry $(k,j)$ of $M_i$, which is the same as saying $m_{ijk}$ is the entry $(n+k,j)$ of the matrix of the inner derivation $\ad v_i$. And, by properties of the inner derivations or of the $n$-quadratic family, we have
\begin{equation*}
	m_{ijk} = m_{jki} = m_{kij} = - m_{ikj} = -m_{jik} = -m_{kji}.
\end{equation*}
Even more, $\varphi([v_i,v_j], v_k)=m_{ijk}$, and the non-degeneration of the family is equivalent~to
\begin{equation*}
	\sum_{i=1}^k \im \ad v_i = A^2 = \spa\langle z_i : i=1,\ldots,n\rangle = Z(A).
\end{equation*}
And, we can recover the Lie product of $A$ from the matrix equation:
\begin{equation*}
	([v_1,v_2],.., [v_1,v_n],[v_2,v_3], ..,[v_2,v_n],.., [v_{n-1},v_n])=\\(z_1,.., z_n)\cdot \mathfrak{F}(M_1,.., M_n).
\end{equation*}
Previous arguments let us reformulate Theorem 8 in \cite{benito2019quadratic} as follows:

\begin{theorem}\label{thm:reducedFamily}
	Let $\{M_1, \ldots, M_n\}$ be a nonzero family of $n$-quadratic matrices of order $n\times n$ with entries in $\mathbb{K}$. On the vector space $\ku^{2n}$ with canonical basis $\{e_1, \dots, e_{2n}\}$, let consider the hyperbolic form $\varphi(e_i, e_{2n-i+1})=\varphi(e_{2n-i+1}, e_i)=1$ and $\varphi(e_i, e_j)=0$ otherwise, and the product given by bilinear extension of the brackets
	\begin{equation*}
		[e_i,e_j]= (e_{n+1}, \dots, e_{2n})\cdot \mathrm{Col}_j\,M_i,
	\end{equation*}
	with $\mathrm{Col}_j\,M_i$ being the $j$-th column of the matrix $M_i$. Then, $(\ku^{2n}, \varphi)$ is a quadratic $2$-step Lie algebra. Moreover, the following conditions are equivalent:
	\begin{enumerate}[\quad (a)]
		\item $(\ku^{2n}, \varphi)$ is reduced,
		\item $\{M_1, \ldots, M_n\}$ is a non-degenerate $n$-quadratic family,
		\item the derived algebra of $\ku^{2n}$ is just $W=\spa\langle e_{n+1}, \dots, e_{2n}\rangle$.
	\end{enumerate}

	\noindent Conversely, any $2$-step quadratic reduced $2n$-dimensional Lie algebra is isometrically isomorphic to $(\ku^{2n}, f)$ for some non-degenerate $n$-quadratic family.
\end{theorem}

\begin{remark}
	For any $n\neq 1,2,4$, there are no non-degenerate $n$-quadratic families according to Proposition 13 in \cite{benito2019quadratic}. But if $n=1,2,4$, there are not. Moreover, \cite[Theorem 14]{benito2019quadratic} solves the problem of isomorphisms of quadratic $2$-step reduced algebras in terms of a matrix relation between the non-degenerate quadratics families attached to them.
\end{remark}

\begin{remark}
	The approach to quadratic $2$-step Lie algebras given by free quadratic families allows us to introduce computational algorithms for building examples of this class of Lie algebras.
\end{remark}


\section{Equivalence theorem}

Despite methods introduced in Section~\ref{sec:methods} are apparently totally different, all three of them end up constructing the same type of algebras and, therefore, it makes sense they are equivalent in some way. The relationship among them appears in the following theorem.

\begin{theorem}\label{thm:equivalence}
	Let $B$ be a vector space with basis $\mathcal{B}=\{b_1, \ldots, b_n\}$, $n\geq 3$ and, let $w\colon B \times B \to B^*$ be a bilinear form where $w(b_i, b_j)(b_k) = c_{ijk}$. In the vector space $\mathfrak{L}=B \oplus B^*$ we define the following product and bilinear form $\phi$ for $b,b' \in B$ and $\beta, \beta' \in B^*$:
	\begin{equation*}
		[b + \beta, b' + \beta'] = w(b,b'), \qquad \phi(b + \beta, b' + \beta') = \beta(b') + \beta'(b).
	\end{equation*}
	Then, it is equivalent:
	\begin{enumerate}[(a)]
		\item $(\mathfrak{L},\phi)$ is a 2-step quadratic Lie algebra. \label{item:equiv_a}
		\item $w$ is a nonzero cyclic 2-cocycle and $(\mathfrak{L},\phi) = (T^*_w B,q_B)$. \label{item:equiv_b}
		\item For $\{b_1^*, \dots, b_n^*\}$ dual basis of $\mathcal{B}$, the chain of one-dimensional double extensions $\{(A_k,f_k)\}_{k=0}^n$ starting with $A_0=\{0\}, f_0=0$ and given by $\{(b_{k+1},d_k)\}_{k=0}^{n-1}$, where $A_k=\spa\langle b_i, b^*_i: i=1, \dots, k\rangle$, $d_{i-1}(b_j^*)=0$, and $d_{i-1}(b_j) = \sum_{k=1}^{i-1} c_{ijk} b_k^*$ for $j< i\leq n$, satisfies properties \emph{(NNP)} and \emph{(2SP)}, and  $(A_n,f_n)=(\mathfrak{L}, \phi)$. \label{item:equiv_c}
		\item The family of matrices $\{M_1, \dots, M_n\}_{i=1}^n$, where the entrance $(k,j)$ of $M_i$ is $c_{ijk}$, is a non-null $n$-quadratic family and defines algebra $(A, \varphi) = (\mathfrak{L}, \phi)$. 
	\end{enumerate}
\end{theorem}
\begin{proof}
	The construction given in this theorem is exactly the $T^*$-extension one, so $(\mathfrak{L}, \phi)$ is $(T^*_w B,q_B)$, the $T^*$-extension of the abelian Lie algebra $B$ by $w$. This proves the equivalence between \eqref{item:equiv_a} and \eqref{item:equiv_b} using Proposition~\ref{prop:cond2coc}.

	Assume now (b) holds and let decompose $A_{k}=A_{k,1}\oplus A_{k,2}$ as in equation~\eqref{eq:decomposition_Ak}. Note that $A_n=\mathfrak{L}$ (as vector spaces) and $f_k:=\phi\big|_{A_k\times A_k}$ are as in Remark~\ref{rmk:chained-bilform}. In particular, $f_n=\phi$. So, the chain will be of one-dimensional double extensions if and only if $d_{i-1}\in \der_{f_{i-1}}A_{i-1}$. Since $B^*$ is $\phi$-isotropic and $d_{i-1}(b_j^*)=0$, this assertion is equivalent to  $\lambda_{ijs}=0$ where:
	\begin{equation*}
		\lambda_{ijs}=f_{i-1}(d_{i-1}(b_j), b_s)+f_{i-1}(b_j, d_{i-1}(b_s))=\varphi(d_{i-1}(b_j), b_s)+\varphi(b_j, d_{i-1}(b_s)).
	\end{equation*}
	Now from $d_{i-1}(A_{i-1})\subseteq A_{{i-1},2}$, $\lambda_{ijs}=0$ if $s\geq i$. Otherwise $\lambda_{ijs}=c_{ijs}+c_{isj}$ and it is also null because of $w$ is cyclic and skew. Moreover since $A_{k,1}=\spa\langle b_1, \ldots, b_k\rangle$ and $ A_{k,2}=\spa\langle b_1^*, \ldots, b_k^*\rangle$ the chain satisfies property (2SP). Finally, from $w\neq 0$ we have that $c_{i_0jk}\neq 0$ for some $i_0$ index, thus $d_{i_0-1}\neq 0$ and the chain satisfies (NNP).
	Next, observe that from Lemma \ref{lem:chain}, the chain described in \eqref{item:equiv_c} ends up in a quadratic Lie algebra $(A_n, f_n)$ such that
	\begin{alignat*}{3}
		[b_i, b_j]_n(b_k) & = D_{ijk}\quad\mathrm{and}\quad & [\beta, \blank]_n=0 \quad \forall \beta\in B^*,
	\end{alignat*}

	where scalars $D_{ijk}$ are defined in that lemma. But $[b_i, b_j]_A(b_k)=c_{ijk}$ and for $j,k<i$, with $j\neq k$ we have
	\begin{align*}
		c_{ijk}=f_{i-1}(d_{i-1}(b_j), b_k) & \underset{_{j<k<i}}{=}D_{jki}=\sgn((i\,j\,k))D_{ijk}=D_{ijk}, \\
		c_{ijk}=f_{i-1}(d_{i-1}(b_j), b_k) & \underset{_{k<j<i}}{=}D_{jki}=\sgn((j\,k\,i))D_{ijk}=D_{ijk}.
	\end{align*}

	Now, from $w$ being cyclic and skew we get $c_{ijk} = \sgn(\sigma)c_{\sigma(i)\sigma(j)\sigma(k)}$ for every permutation $\sigma$ and $c_{iik}=0$. So we have $[b_i, b_j]_A(b_k)=c_{ijk}=D_{ijk}=[b_i, b_j]_n(b_k)$. Hence, $[\blank, \blank]_n=[\blank, \blank]_A$ and $(\mathfrak{L},\phi)=(A_n, f_n)$ as quadratic Lie algebras.

	To prove that (c) implies (d) just apply Lemma \ref{lem:chain} taking into account that $\sgn(\sigma)D_{\sigma(i)\sigma(j)\sigma(k)}=D_{ijk}=c_{ijk}$ and $[b_i, b_j]_n=\sum_{k=1}^nc_{ijk}b_k^*$. So, the entry $c_{ijk}$ of every matrix $M_i$ described in (d) is the entry in the position $(n+k,j)$ of the matrix of the inner derivation $\ad b_i$. This proves the matrix family is $n$-quadratic. Finally, the definition of $n$-quadratic family yields to $w$ being a nonzero cyclic $2$-cocycle.
\end{proof}

Note that, with this theorem, we can also check the reduced conditions required in each method are equivalent among them.

Once at this point, we are going to see that we can move easily between the three methods directly from their respective constructions previously given in this paper.

\begin{itemize}
	\item In a chain $\{(A_k,f_k)\}_{k=1}^n$ of one-dimensional double extensions we consider a basis $\{b_{n},\ldots,b_1,b_1^*,\ldots,b_{n}^*\}$ as before. Then
	      \begin{alignat*}{3}
		      [b_i, b_j]   & = \sum_{k=1}^{n} D_{ijk}e_k^*,\qquad & [b_i^*, \cdot\,] & = 0,           \\
		      f_n(b_i,b_j) & = f_n(b_i^*,b_j^*) = 0,\qquad        & f_n(b_i,b_j^*)   & = \delta_{ij}.
	      \end{alignat*}

	\item In $(T^*_wB,q_B)$ with basis $\{e_1,\ldots, e_n, e_1^*, \ldots, e_n^*\}$
	      \begin{equation}\label{mth:tw}
		      \begin{alignedat}{3}
			      [e_i, e_j] &= w(e_i,e_j) = \sum_{k=1}^{n} w_{ijk}e_k^*,\qquad  &[e_i^*, \cdot\,] &= 0,\\
			      q_B(e_i,e_j) &= q_B(e_i^*,e_j^*) = 0,\qquad   &q_B(e_i,e_j^*) &= \delta_{ij}.
		      \end{alignedat}
	      \end{equation}

	\item An $n$-quadratic family $\{M_1, \ldots, M_n\}$ where $m_{ijk}$ is the entry $(j,k)$ of $M_i$ defines over the basis $\{v_1,\ldots,v_n, z_1, \ldots, z_n\}$ of $A$ the quadratic Lie algebra $(A,\varphi)$ where
	      \begin{alignat*}{3}
		      [v_i, v_j]       & = \sum_{k=1}^{n} m_{ijk}z_k,\qquad & [z_i, \cdot\,]   & = 0,           \\
		      \varphi(v_i,v_j) & = \varphi(z_i,z_j) = 0,\qquad      & \varphi(v_i,z_j) & = \delta_{ij}.
	      \end{alignat*}
\end{itemize}
So the equivalence comes from just renaming
\begin{equation*}
	\begin{array}{rcccl}
		b_{i}   & \longleftrightarrow & e_i     & \longleftrightarrow & v_i,     \\
		b_{i}^* & \longleftrightarrow & e_i^*   & \longleftrightarrow & z_i,     \\
		f_n     & \longleftrightarrow & q_B     & \longleftrightarrow & f,       \\
		D_{ijk} & \longleftrightarrow & w_{ijk} & \longleftrightarrow & m_{ijk}.
	\end{array}
\end{equation*}

Therefore, if we have a $n$-quadratic family of matrices with coefficients $m_{ijk}$ we can define the equivalent $(T_w^*B, q_B)$ extension taking
\begin{equation*}
	w(b_i,b_j)(b_k) = w_{ijk} = m_{ijk}.
\end{equation*}
And we can also obtain a chain of one-dimensional double extensions taking
\begin{align*}
	d_{i-1}\colon A_{i-1} & \to A_{i-1}                                                                      \\
	b_j                   & \mapsto \sum_{k=1}^{i-1}w_{ijk}\,b_k^* = \sum_{k=1}^{i-1}w(e_i,e_j)(e_k)\,b_k^*, \\
	b_j^*                 & \mapsto 0.
\end{align*}
And vice versa, if we have built a chain, we can get the equivalent $T_w^*B$ if we take
\begin{equation*}
	w(e_i,e_j) = \sum_{k=1}^{n}D_{ijk}\,e_k^* = \sum_{k=1}^{n}f_n([b_i,b_j],b_k)\,e_k^*.
\end{equation*}
And, this also defines our $n$-quadratic family of matrices taking $m_{ijk} = D_{ijk}$.

\begin{example}
	When we try to generate a generic 2-step quadratic Lie algebra of dimension $n$ we end up with $n(n-2)(n-4)/48$ parameters, a number that grows pretty fast. Even in the 10-dimensional algebra we have 10 parameters in its general form, despite all of them are isometrically isomorphic. These parameters can be observed in any of the three equivalent constructions. For example, when constructing a (2SP) and (NNP) chain of one-dimensional double extensions we obtain the following derivations:
	\begin{gather*}
		d_0 = d_1 = 0, d_2 = \left(
		\begin{array}{cc|c}
				\multicolumn{2}{c|}{\bm{0}} & \bm{0}  \\\hline
				-D_{123} & 0 & \multirow{2}{*}{$\bm{0}$} \\
				0 & D_{123}\\
			\end{array}
		\right),
		d_3 = \left(
		\begin{array}{ccc|c}
				\multicolumn{3}{c|}{\bm{0}} & \bm{0} \\\hline
				-D_{134} & -D_{124} & 0 & \multirow{3}{*}{$\bm{0}$} \\
				-D_{234} & 0 & D_{124}\\
				0 & D_{234} & D_{134}\\
			\end{array}
		\right),\\
		d_4 =
		\left(
		\begin{array}{cccc|c}
				\multicolumn{4}{c|}{\bm{0}} & \bm{0} \\\hline
				-D_{145} & -D_{135} & -D_{125} & 0 & \multirow{4}{*}{$\bm{0}$} \\
				-D_{245} & -D_{235} & 0 & D_{125} \\
				-D_{345} & 0 & D_{235} & D_{135} \\
				0 & D_{345} & D_{245} & D_{145}\\
			\end{array}
		\right)
	\end{gather*}
	These same parameters appear in $T^*$-extensions or in a more condensed way in $5$-quadratic families, where $\mathfrak{F}(M_1,\ldots,M_5)$ is
		{\small\begin{equation*}
				\begin{pmatrix}
					0       & 0        & 0        & 0        & m_{123} & m_{124}  & m_{125}  & m_{134} & m_{135}  & m_{145} \\
					0       & -m_{123} & -m_{124} & -m_{125} & 0       & 0        & 0        & m_{234} & m_{235}  & m_{245} \\
					m_{123} & 0        & -m_{134} & -m_{135} & 0       & -m_{234} & -m_{235} & 0       & 0        & m_{345} \\
					m_{124} & m_{134}  & 0        & -m_{145} & m_{234} & 0        & -m_{245} & 0       & -m_{345} & 0       \\
					m_{125} & m_{135}  & m_{145}  & 0        & m_{235} & m_{245}  & 0        & m_{345} & 0        & 0       \\
				\end{pmatrix}
			\end{equation*}}
	All this complexity in terms of classification can be reduced using the next section. For instance, we will see all algebras in this example are isometrically isomorphic to the one where $D_{123} = D_{145} = 1$ or $m_{123} = m_{145} = 1$ and the rest of the entries are zero, named as $\mathfrak{L}_{5,1}$ in the following section.
\end{example}

\section{Trivectors and 2-step classification}

In this section, we follow the main ideas given in \cite[Section 3]{noui1997algebres}. For basic notions on multilinear algebra see \cite[Appendix~B]{fulton2013representation}.

The classification of quadratic $2$-step Lie algebras of dimension $2n$ can be reduced to that of trilinear alternating forms or trivectors over a $n$-dimensional vector space $V$. In this section we will explain \emph{why and how} under the scope of previous construction methods of quadratic algebras. We point out that, whereas the problem of classifying bilinear alternating forms is elemental, the classification of trivectors seems tractable only for small values of $n$.

Let $\{e_1, \dots, e_n\}$ be a basis of $V$. The exterior power $\Lambda^m V$ or $\alt^m V$ is a vector space associated to a universal alternating multilinear form
\begin{align*}
	\wedge\colon V \times \dots \times V & \to \Lambda^m V                       \\
	(v_1,\ldots, v_m)                    & \mapsto v_1 \wedge \ldots \wedge v_m.
\end{align*}
The dimension of $\Lambda^m V$ is $\binom{m}{n}$, and $\{e_{i_1} \wedge \cdots \wedge e_{i_m} : 1\leq i_1 < \ldots < i_m\leq n\}$ is its standard basis. Every element of $\Lambda^m V$ is called a $m$-vector. So a trivector is simply an element of $\Lambda^3 V$. Therefore, every trivector can be expressed as a linear combination of their corresponding basis $\{e_i\wedge e_j \wedge e_k : 1\leq i<j<k\leq 3\}$.

If $V^*$ is the dual space of $V$, $\varphi_i\in V^*, v_i\in V$, the map $\iota\colon \Lambda^m V^*\to (\Lambda^m V)^*$ given explicitly as
\begin{align*}
	(\varphi_1 \wedge \ldots \wedge \varphi_m) & \mapsto
	\left(
	v_1 \wedge \ldots \wedge v_m
	\mapsto
	\sum_{\sigma \in S_m} \sgn(\sigma) \prod_{i=1}^{m} \varphi_{\sigma(i)}(v_i)  = \det{(\varphi_j(v_i))}
	\right)
\end{align*}
is an  isomorphism. The elements of $(\Lambda^m V)^*$ are named $m$-alternating forms or $m$-forms.  We also note that $\iota^{-1}$ sends the linear form $(e_{i_1} \wedge \ldots \wedge e_{i_m})^*$ back into $e_{i_1}^* \wedge \ldots \wedge e_{i_m}^*$. Since $(\Lambda^m V)^*$ is isomorphic to $\Lambda^m V$, there is no difference between $m$-vectors and $m$-forms.

Now that we know what a trivector is, we can see its relationship with reduced quadratic $2$-step Lie algebras. In order to see it, we can use $T^*$-extensions of abelian Lie algebras as mentioned in Corollary~\ref{cor:duong}.
Here, every algebra $T^*_w  B$ obtained of the same dimension differs only in the mapping $w\colon B \times B \to B^*$, as $B$ is an abelian algebra and the bilinear form is defined in the same way.
At this point, we can define $\phi_w\colon \Lambda^3B\to \mathbb{K}$ taking $\phi_w(b_1,b_2,b_3) = w(b_1,b_2)(b_3)$.
Note $\phi_w \in (\Lambda^3B)^* \cong \Lambda^3 B^*$ is a trivector thanks to the bilinear map $w$ is cyclic, so satisfies equation~\eqref{eq:cyclic}, and skew-symmetric (see Remark~\ref{rmk:2cocyc_3cocyc}).
In fact, the set $\{w_{ijk}=w(e_i, e_j)(e_k) : i < j < k\}$ seen in equation~\eqref{mth:tw} is simply the coordinates of the trivector $\phi_w$ in the standard dual basis, $(e_i\wedge e_j\wedge e_k)^*\sim e_i^*\wedge e_j^*\wedge e_k^*$.
Therefore, every quadratic 2-step Lie algebra can be defined from a trivector and also a quadratic $2$-step Lie algebra gives us a trivector. In this way, we arrive at the bijection
\begin{equation}\label{eq:biy_coc_tri}
	\Delta\colon\{w\in Z^2(B, B^*): w \ \mathrm{is\ cyclic}\}\to (\Lambda^3B)^*, \quad
	w \mapsto \phi_w
\end{equation}
given by the expression $w(e_i,e_j)(e_k) = \phi_w(e_i,e_j,e_k)$.

Even more, $\ker \phi_w=\{x\in B: \phi_w(x, \blank, \blank)=0\}=\rad w$. Thus $\Delta$ sends a non-degenerate $w$ into a trivector $\phi_w$ such that $\ker \phi_w=0$, and conversely. The nullity of $\ker \phi_w$ is equivalent to say that $\phi_w$ is a trivector of (maximal) rank equal to $\dim B$. Following \cite{cohen1988trilinear}, the rank of a trivector $\phi \in \Lambda^3 V$ is $\rank \phi=\dim V-\dim \ker \phi$. The rank of $\phi$ agrees with the dimension of the smallest subspace $W$ of $V$ such that $\phi \in \Lambda^3W$ (see \cite[Section 1]{noui1994formes}).

But the important point is that not only a bijection between quadratic $2$-step Lie algebras and trivectors exits. It is the fact the bijection maps isometrically isomorphic $2$-step $T^*$-extensions into equivalent trivectors with respect to the natural equivalence relation given by the action of the general linear group (see Definition~\ref{def:equivalentTrivector}).

\begin{definition}\label{def:equivalentTrivector}
	We say two trivectors $\phi_1,\, \phi_2 \in \Lambda^3 V$ are equivalent if there exist $\sigma \in \text{GL}(V)$ such that  $\phi_1(x,y,z)=\phi_2(\sigma(x),\sigma(y),\sigma(z))$ for every $x,y,z \in V$. Hence $\sigma \cdot \phi_1=\phi_2$, letting $\sigma$ act on the trivectors by means of $(\sigma\cdot \phi)(x,y,z)=\phi(\sigma^{-1}(x),\sigma^{-1}(y),\sigma^{-1}(z))$.
\end{definition}

\begin{theorem}\label{thm:trivector-classification}
	Let $B$ be a Lie algebra and $B^*$ its coadjoint module, $w, w_1, w_2\in Z^2(B,B^*)$ and cyclic. The map $\Delta$ defined in equation~\eqref{eq:biy_coc_tri} is an involutive bijection satisfying the following properties:
	\begin{enumerate}[\quad (a)]
		\item $w$ is non-degenerate if and only if $\rank (\phi_w)=\dim B$.
		\item If $B$ is abelian and $w_1, w_2$ are non-degenerate, $T_{w_1}^*B$ and $T_{w_2}^*B$ are isometrically isomorphic if and only if $\phi_{w_1}$ and $\phi_{w_2}$ are equivalent trivectors.
	\end{enumerate}
\end{theorem}
\begin{proof}
	For arbitrary $B$, $\Delta$ is well defined according to Remark \ref{rmk:2cocyc_3cocyc}. Thus (a) follows from previous comments. Before proving item (b), we recall that Lie bracket of $T^*_w B=B\overset{w}{\oplus} B^*$ is given by $[a+\alpha,b+\beta]_w=w(a,b)$ if $B$ is abelian and, from Corollary \ref{cor:reduced}, $Z(T^*_w B)=B^*$ if $w$ is non-degenerate. Hence, assuming $B$ abelian and $w_1, w_2$ non-degenerate, for a given isometrically isomorphism  $\varphi$ from $T_{w_1}^*B$ onto $T_{w_2}^*B$, we have $\varphi(Z(T_{w_1}))=Z(T_{w_2})=B^*$, thus $T_{w_2}^*B=B^*\overset{w_2}{\oplus} \varphi(B)=B^*\overset{w_2}{\oplus}B$. This implies that $\sigma=\pi_B\circ \varphi\in \text{GL}(B)$, where $\pi_B$ is the projection map from $T_{w_2}^*B$ onto $B$. Then,
	\begin{multline*}
		\phi_{w_1}(x,y,z)=w_1(x,y)(z)=f([x,y]_{w_1},z)=f([\varphi(x),\varphi(y)]_{w_2},\varphi(z))=\\
		f([\sigma(x),\sigma(y)]_{w_2},\sigma(z)) =w_2(\sigma(x),\sigma(y)(\sigma(z)) = \phi_2(\sigma(x),\sigma(y),\sigma(z)).
	\end{multline*}
	Thus $\phi_{w_1}$ and $\phi_{w_1}$ are equivalent. On the contrary, if $\sigma \in \text{GL}(B)$ such that $\phi_{w_1}(x,y,z) = \phi_{w_2}(\sigma(x),\sigma(y),\sigma(z))$,  $w_1(b_1,b_2)\circ \sigma^{-1}=w_2(\sigma(b_1), \sigma(b_2))$ follows easily, and the map
	\begin{align*}
		\varphi\colon T_{w_1}B = B \oplus B^* & \to T_{w_2}B = B \oplus B^*                 \\
		b + \beta                             & \mapsto \sigma(b) + \beta \circ \sigma^{-1}
	\end{align*}
	is an isometric isomorphism.
\end{proof}

\begin{corollary}
	The map $\Delta$ defined in equation~\eqref{eq:biy_coc_tri} provides a natural bijection between isomorphism classes of reduced quadratic 2-step nilpotent Lie algebras of dimension $2n$ and the equivalence classes of trivectors of rank $n$.
\end{corollary}

This result has been established in~\cite[3.5 Th\'eor\`eme]{noui1997algebres} and it is quite useful as classification tables for trivectors are available. In order to simplify notation, from now on, trivector $e_i^*\wedge e_j^* \wedge e_k^*$ will be denoted as $ijk$. So $123+456 \longleftrightarrow e_1^*\wedge e_2^* \wedge e_3^*+e_4^*\wedge e_5^* \wedge e_6^*$.

The fact each quadric reduced $2$-step Lie algebra can be associated to a trivector and vice versa means their classification, thanks to Theorem~\ref{thm:trivector-classification}, is equivalent to the trivectors one. This allows us to obtain a list of these algebras, as trivectors have been already classified for low dimensions. A nice classification up to dimension 9, over the complex field $\mathbb{C}$, appears in \cite{Vinberg_Elashvili_1988} by using a $\mathbb{Z}_3$-grading of the simple Lie algebra $\mathfrak{e}_8$. Cohen and Helminck (see~\cite{cohen1988trilinear}) classify trivectors up to dimension 7 over fields of cohomological dimension at most $1$, which includes algebraically closed fields and finite fields.

Over the complex field, we can know how many reduced quadratic 2-step Lie algebras are there up to isometrically isomorphisms using less than 9 generators. This data is showed in Table~\ref{tab:2step}, where the dimension $2n\leq 18$ of the Lie algebra is related to the rank of the corresponding trivector, which is just $n\leq 9$.
\begin{table}
	\tbl{Non-isometric reduced quadratic 2-step Lie algebras in $\mathbb{C}$ (source~\cite{Vinberg_Elashvili_1988}).}
	{\begin{tabular}{p{0.1\textwidth}*{7}{>{\centering\arraybackslash}p{0.06\textwidth}}}\toprule
			Dimension & 6 & 8 & 10 & 12 & 14 & 16 & $\geq 18$ \\\midrule
			Number    & 1 & 0 & 1  & 2  & 5  & 13 & $\infty$  \\\bottomrule
		\end{tabular}}
	\label{tab:2step}
\end{table}

Moreover, we are also able to give a representative of each of these algebras and find its multiplication table. Along the following list we consider a quadratic algebra $(A, f)$ of dimension $2n$ with basis $\{e_1, \ldots, e_n, e_1^*, \ldots, e_n^*\}$, where $A^2 = \spa\langle e_1^*, \ldots, e_n^*\rangle$, $f(e_i,e_j) = f(e_i^*,e^*_j) = 0$, and $f(e_i, e_j^*) = \delta_{ij}$. Each algebra receives a name of the form $\fL_{n,k}$, where $n$ is the type, also half the dimension, and $k$ is the position it occupies in the list among all algebras of the same type/dimension. In addition, to simplify the list we only show non-zero products of the form $[e_i, e_j]$ where $i<j$. According to the map $\Delta$ described in equation~\eqref{eq:biy_coc_tri}, the rule to display the different multiplication tables is given by the coordinates of the trivectors $\phi_w=\sum w_{ijk}e_i^*\wedge e_j^*\wedge e_k^*$, so
\begin{equation}\label{eq:process_trivector_algebra}
	[e_i,e_j]=w_{ijk}e_k^* \Longleftrightarrow \phi_w(e_i,e_j,e_k)=w_{ijk}=w(e_i,e_j)(e_k).
\end{equation}
In this way, any $2$-step quadratic reduced Lie algebras of dimension less than 17 over the complex field up to isometric isomorphisms is given in the following list:

\begin{itemize}
	\item One 6-dimensional algebra:
	      \begin{itemize}
		      \item Algebra $\fL_{3,1}$ associated to trivector $123$:
		            \begin{alignat*}{3}
			            [e_1,e_2] & = e^*_3,\qquad & [e_1,e_3] & = -e^*_2,\qquad & [e_2,e_3] & = e^*_1.
		            \end{alignat*}
	      \end{itemize}

	\item One 10-dimensional algebra:
	      \begin{itemize}
		      \item Algebra $\fL_{5,1}$ associated to trivector $123+145$:
		            \begin{alignat*}{3}
			            [e_1,e_2] & = e^*_3,\qquad  & [e_1,e_3] & = -e^*_2,\qquad & [e_1,e_4] & = e^*_5, \\
			            [e_1,e_5] & = -e^*_4,\qquad & [e_2,e_3] & = e^*_1,\qquad  & [e_4,e_5] & = e^*_1.
		            \end{alignat*}
	      \end{itemize}

	\item Two 12-dimensional algebras:
	      \begin{itemize}
		      \item Algebra $\fL_{6,1}$ associated to trivector $123+456$:
		            \begin{alignat*}{3}
			            [e_1,e_2] & = e^*_3,\qquad & [e_1,e_3] & = -e^*_2,\qquad & [e_2,e_3] & = e^*_1, \\
			            [e_4,e_5] & = e^*_6,\qquad & [e_4,e_6] & = -e^*_5,\qquad & [e_5,e_6] & = e^*_4.
		            \end{alignat*}

		      \item Algebra $\fL_{6,2}$ associated to trivector $124+135+236$:
		            \begin{alignat*}{3}
			            [e_1,e_2] & = e^*_4,\qquad  & [e_1,e_3] & = e^*_5,\qquad & [e_1,e_4] & = -e^*_2, \\
			            [e_1,e_5] & = -e^*_3,\qquad & [e_2,e_3] & = e^*_6,\qquad & [e_2,e_4] & = e^*_1,  \\
			            [e_2,e_6] & = -e^*_3,\qquad & [e_3,e_5] & = e^*_1,\qquad & [e_3,e_6] & = e^*_2.
		            \end{alignat*}
	      \end{itemize}

	\item Five 14-dimensional algebras:
	      \begin{itemize}
		      \item Algebra $\fL_{7,1}$ associated to trivector $123+145+167$:
		            \begin{alignat*}{3}
			            [e_1,e_2] & = e^*_3,\qquad  & [e_1,e_3] & = -e^*_2,\qquad & [e_1,e_4] & = e^*_5,  \\
			            [e_1,e_5] & = -e^*_4,\qquad & [e_1,e_6] & = e^*_7,\qquad  & [e_1,e_7] & = -e^*_6, \\
			            [e_2,e_3] & = e^*_1,\qquad  & [e_4,e_5] & = e^*_1,\qquad  & [e_6,e_7] & = e^*_1.
		            \end{alignat*}

		      \item Algebra $\fL_{7,2}$ associated to trivector $127+134+256$:
		            \begin{alignat*}{3}
			            [e_1,e_2] & = e^*_7,\qquad  & [e_1,e_3] & = e^*_4,\qquad & [e_1,e_4] & = -e^*_3, \\
			            [e_1,e_7] & = -e^*_2,\qquad & [e_2,e_5] & = e^*_6,\qquad & [e_2,e_6] & = -e^*_5, \\
			            [e_2,e_7] & = e^*_1,\qquad  & [e_3,e_4] & = e^*_1,\qquad & [e_5,e_6] & = e^*_2.
		            \end{alignat*}

		      \item Algebra $\fL_{7,3}$ associated to trivector $125+136+147+234$:
		            \begin{alignat*}{4}
			            [e_1,e_2] & = e^*_5, \qquad & [e_1,e_3] & = e^*_6, \qquad & [e_1,e_4] & = e^*_7, \qquad & [e_1,e_5] & = -e^*_2, \\
			            [e_1,e_6] & = -e^*_3,\qquad & [e_1,e_7] & = -e^*_4,\qquad & [e_2,e_3] & = e^*_4, \qquad & [e_2,e_4] & = -e^*_3, \\
			            [e_2,e_5] & = e^*_1, \qquad & [e_3,e_4] & = e^*_2, \qquad & [e_3,e_6] & = e^*_1, \qquad & [e_4,e_7] & = e^*_1.
		            \end{alignat*}

		      \item Algebra $\fL_{7,4}$ associated to trivector $125+137+247+346$:
		            \begin{alignat*}{4}
			            [e_1,e_2] & = e^*_5, \qquad & [e_1,e_3] & = e^*_7,\qquad & [e_1,e_5] & = -e^*_2,\qquad & [e_1,e_7] & = -e^*_3, \\
			            [e_2,e_4] & = e^*_7, \qquad & [e_2,e_5] & = e^*_1,\qquad & [e_2,e_7] & = -e^*_4,\qquad & [e_3,e_4] & = e^*_6,  \\
			            [e_3,e_6] & = -e^*_4,\qquad & [e_3,e_7] & = e^*_1,\qquad & [e_4,e_6] & = e^*_3, \qquad & [e_4,e_7] & = e^*_2.
		            \end{alignat*}

		      \item Algebra $\fL_{7,5}$ associated to trivector $123+147+257+367+456$:
		            \begin{alignat*}{4}
			            [e_1,e_2] & = e^*_3, \qquad & [e_1,e_3] & = -e^*_2,\qquad & [e_1,e_4] & = e^*_7, \qquad & [e_1,e_7] & = -e^*_4, \\
			            [e_2,e_3] & = e^*_1, \qquad & [e_2,e_5] & = e^*_7, \qquad & [e_2,e_7] & = -e^*_5,\qquad & [e_3,e_6] & = e^*_7,  \\
			            [e_3,e_7] & = -e^*_6,\qquad & [e_4,e_5] & = e^*_6, \qquad & [e_4,e_6] & = -e^*_5,\qquad & [e_4,e_7] & = e^*_1,  \\
			            [e_5,e_6] &= e^*_4, \qquad &[e_5,e_7] &= e^*_2, \qquad &[e_6,e_7] &= e^*_3.
		            \end{alignat*}
	      \end{itemize}

	\item Thirteen 16-dimensional algebras:
	      \begin{itemize}
		      \item Algebra $\fL_{8,1}$ associated to trivector $156+178+234$:
		            \begin{alignat*}{3}
			            [e_1,e_5] & = e^*_6,\qquad  & [e_1,e_6] & = -e^*_5,\qquad & [e_1,e_7] & = e^*_8,  \\
			            [e_1,e_8] & = -e^*_7,\qquad & [e_2,e_3] & = e^*_4,\qquad  & [e_2,e_4] & = -e^*_3, \\
			            [e_3,e_4] & = e^*_2,\qquad  & [e_5,e_6] & = e^*_1,\qquad  & [e_7,e_8] & = e^*_1.
		            \end{alignat*}

		      \item Algebra $\fL_{8,2}$ associated to trivector $127+138+145+236$:
		            \begin{alignat*}{4}
			            [e_1,e_2] & = e^*_7, \qquad & [e_1,e_3] & = e^*_8, \qquad & [e_1,e_4] & = e^*_5,\qquad & [e_1,e_5] & = -e^*_4, \\
			            [e_1,e_7] & = -e^*_2,\qquad & [e_1,e_8] & = -e^*_3,\qquad & [e_2,e_3] & = e^*_6,\qquad & [e_2,e_6] & = -e^*_3, \\
			            [e_2,e_7] & = e^*_1, \qquad & [e_3,e_6] & = e^*_2, \qquad & [e_3,e_8] & = e^*_1,\qquad & [e_4,e_5] & = e^*_1.
		            \end{alignat*}

		      \item Algebra $\fL_{8,3}$ associated to trivector $125+137+248+346$:
		            \begin{alignat*}{4}
			            [e_1,e_2] & = e^*_5, \qquad & [e_1,e_3] & = e^*_7,\qquad & [e_1,e_5] & = -e^*_2,\qquad & [e_1,e_7] & = -e^*_3, \\
			            [e_2,e_4] & = e^*_8, \qquad & [e_2,e_5] & = e^*_1,\qquad & [e_2,e_8] & = -e^*_4,\qquad & [e_3,e_4] & = e^*_6,  \\
			            [e_3,e_6] & = -e^*_4,\qquad & [e_3,e_7] & = e^*_1,\qquad & [e_4,e_6] & = e^*_3, \qquad & [e_4,e_8] & = e^*_2.
		            \end{alignat*}

		      \item Algebra $\fL_{8,4}$ associated to trivector $137+168+236+245$:
		            \begin{alignat*}{4}
			            [e_1,e_3] & = e^*_7,\qquad & [e_1,e_6] & = e^*_8,\qquad & [e_1,e_7] & = -e^*_3,\qquad & [e_1,e_8] & = -e^*_6, \\
			            [e_2,e_3] & = e^*_6,\qquad & [e_2,e_4] & = e^*_5,\qquad & [e_2,e_5] & = -e^*_4,\qquad & [e_2,e_6] & = -e^*_3, \\
			            [e_3,e_6] & = e^*_2,\qquad & [e_3,e_7] & = e^*_1,\qquad & [e_4,e_5] & = e^*_2, \qquad & [e_6,e_8] & = e^*_1.
		            \end{alignat*}

		      \item Algebra $\fL_{8,5}$ associated to trivector $134+178+256+278$:
		            \begin{alignat*}{4}
			            [e_1,e_3] & = e^*_4,\qquad & [e_1,e_4] & = -e^*_3,\qquad & [e_1,e_7] & = e^*_8,\qquad & [e_1,e_8] & = -e^*_7, \\
			            [e_2,e_5] & = e^*_6,\qquad & [e_2,e_6] & = -e^*_5,\qquad & [e_2,e_7] & = e^*_8,\qquad & [e_2,e_8] & = -e^*_7, \\
			            [e_3,e_4] &= e^*_1,\qquad &[e_5,e_6] &= e^*_2, \qquad &[e_7,e_8] &= e^*_1 + e^*_2.
		            \end{alignat*}

		      \item Algebra $\fL_{8,6}$ associated to trivector $128+135+147+237+246$:
		            \begin{alignat*}{4}
			            [e_1,e_2] & = e^*_8, \qquad & [e_1,e_3] & = e^*_5, \qquad & [e_1,e_4] & = e^*_7, \qquad & [e_1,e_5] & = -e^*_3, \\
			            [e_1,e_7] & = -e^*_4,\qquad & [e_1,e_8] & = -e^*_2,\qquad & [e_2,e_3] & = e^*_7, \qquad & [e_2,e_4] & = e^*_6,  \\
			            [e_2,e_6] & = -e^*_4,\qquad & [e_2,e_7] & = -e^*_3,\qquad & [e_2,e_8] & = e^*_1, \qquad & [e_3,e_5] & = e^*_1,  \\
			            [e_3,e_7] &= e^*_2, \qquad &[e_4,e_6] &= e^*_2, \qquad &[e_4,e_7] &= e^*_1.
		            \end{alignat*}

		      \item Algebra $\fL_{8,7}$ associated to trivector $127+138+156+246+345$:
		            \begin{alignat*}{4}
			            [e_1,e_2] & = e^*_7, \qquad & [e_1,e_3] & = e^*_8, \qquad & [e_1,e_5] & = e^*_6, \qquad & [e_1,e_6] & = -e^*_5, \\
			            [e_1,e_7] & = -e^*_2,\qquad & [e_1,e_8] & = -e^*_3,\qquad & [e_2,e_4] & = e^*_6, \qquad & [e_2,e_6] & = -e^*_4, \\
			            [e_2,e_7] & = e^*_1, \qquad & [e_3,e_4] & = e^*_5, \qquad & [e_3,e_5] & = -e^*_4,\qquad & [e_3,e_8] & = e^*_1,  \\
			            [e_4,e_5] &= e^*_3, \qquad &[e_4,e_6] &= e^*_2, \qquad &[e_5,e_6] &= e^*_1.
		            \end{alignat*}

		      \item Algebra $\fL_{8,8}$ associated to trivector $136+158+247+258+345$:
		            \begin{alignat*}{4}
			            [e_1,e_3] & = e^*_6, \qquad & [e_1,e_5] & = e^*_8, \qquad & [e_1,e_6] & = -e^*_3,\qquad & [e_1,e_8] & = -e^*_5, \\
			            [e_2,e_4] & = e^*_7, \qquad & [e_2,e_5] & = e^*_8, \qquad & [e_2,e_7] & = -e^*_4,\qquad & [e_2,e_8] & = -e^*_5, \\
			            [e_3,e_4] & = e^*_5, \qquad & [e_3,e_5] & = -e^*_4,\qquad & [e_3,e_6] & = e^*_1, \qquad & [e_4,e_5] & = e^*_3,  \\
			            [e_4,e_7] &= e^*_2, \qquad &[e_5,e_8] &= e^*_1 + e^*_2.
		            \end{alignat*}

		      \item Algebra $\fL_{8,9}$ associated to trivector $145+167+238+246+357$:
		            \begin{alignat*}{4}
			            [e_1,e_4] & = e^*_5,\qquad & [e_1,e_5] & = -e^*_4,\qquad & [e_1,e_6] & = e^*_7, \qquad & [e_1,e_7] & = -e^*_6, \\
			            [e_2,e_3] & = e^*_8,\qquad & [e_2,e_4] & = e^*_6, \qquad & [e_2,e_6] & = -e^*_4,\qquad & [e_2,e_8] & = -e^*_3, \\
			            [e_3,e_5] & = e^*_7,\qquad & [e_3,e_7] & = -e^*_5,\qquad & [e_3,e_8] & = e^*_2, \qquad & [e_4,e_5] & = e^*_1,  \\
			            [e_4,e_6] &= e^*_2,\qquad &[e_5,e_7] &= e^*_3, \qquad &[e_6,e_7] &= e^*_1.
		            \end{alignat*}

		      \item Algebra $\fL_{8,10}$ associated to trivector $128+167+236+247+345$:
		            \begin{alignat*}{4}
			            [e_1,e_2] & = e^*_8, \qquad & [e_1,e_6] & = e^*_7, \qquad & [e_1,e_7] & = -e^*_6,\qquad & [e_1,e_8] & = -e^*_2, \\
			            [e_2,e_3] & = e^*_6, \qquad & [e_2,e_4] & = e^*_7, \qquad & [e_2,e_6] & = -e^*_3,\qquad & [e_2,e_7] & = -e^*_4, \\
			            [e_2,e_8] & = e^*_1, \qquad & [e_3,e_4] & = e^*_5, \qquad & [e_3,e_5] & = -e^*_4,\qquad & [e_3,e_6] & = e^*_2,  \\
			            [e_4,e_5] &= e^*_3, \qquad &[e_4,e_7] &= e^*_2, \qquad &[e_6,e_7] &= e^*_1.
		            \end{alignat*}

		      \item Algebra $\fL_{8,11}$ associated to trivector $128+136+157+247+256+345$:
		            \begin{alignat*}{4}
			            [e_1,e_2] & = e^*_8, \qquad & [e_1,e_3] & = e^*_6, \qquad & [e_1,e_5] & = e^*_7, \qquad & [e_1,e_6] & = -e^*_3, \\
			            [e_1,e_7] & = -e^*_5,\qquad & [e_1,e_8] & = -e^*_2,\qquad & [e_2,e_4] & = e^*_7, \qquad & [e_2,e_5] & = e^*_6,  \\
			            [e_2,e_6] & = -e^*_5,\qquad & [e_2,e_7] & = -e^*_4,\qquad & [e_2,e_8] & = e^*_1, \qquad & [e_3,e_4] & = e^*_5,  \\
			            [e_3,e_5] & = -e^*_4,\qquad & [e_3,e_6] & = e^*_1, \qquad & [e_4,e_5] & = e^*_3, \qquad & [e_4,e_7] & = e^*_2,  \\
			            [e_5,e_6] &= e^*_2, \qquad &[e_5,e_7] &= e^*_1.
		            \end{alignat*}

		      \item Algebra $\fL_{8,12}$ associated to trivector $126+158+238+257+347+456$:
		            \begin{alignat*}{4}
			            [e_1,e_2] & = e^*_6, \qquad & [e_1,e_5] & = e^*_8, \qquad & [e_1,e_6] & = -e^*_2,\qquad & [e_1,e_8] & = -e^*_5, \\
			            [e_2,e_3] & = e^*_8, \qquad & [e_2,e_5] & = e^*_7, \qquad & [e_2,e_6] & = e^*_1, \qquad & [e_2,e_7] & = -e^*_5, \\
			            [e_2,e_8] & = -e^*_3,\qquad & [e_3,e_4] & = e^*_7, \qquad & [e_3,e_7] & = -e^*_4,\qquad & [e_3,e_8] & = e^*_2,  \\
			            [e_4,e_5] & = e^*_6, \qquad & [e_4,e_6] & = -e^*_5,\qquad & [e_4,e_7] & = e^*_3, \qquad & [e_5,e_6] & = e^*_4,  \\
			            [e_5,e_7] &= e^*_2, \qquad &[e_5,e_8] &= e^*_1.
		            \end{alignat*}

		      \item Algebra $\fL_{8,13}$ associated to trivector $123+178+257+368+456+478$:
		            \begin{alignat*}{4}
			            [e_1,e_2] & = e^*_3, \qquad & [e_1,e_3] & = -e^*_2,\qquad & [e_1,e_7] & = e^*_8, \qquad & [e_1,e_8] & = -e^*_7, \\
			            [e_2,e_3] & = e^*_1, \qquad & [e_2,e_5] & = e^*_7, \qquad & [e_2,e_7] & = -e^*_5,\qquad & [e_3,e_6] & = e^*_8,  \\
			            [e_3,e_8] & = -e^*_6,\qquad & [e_4,e_5] & = e^*_6, \qquad & [e_4,e_6] & = -e^*_5,\qquad & [e_4,e_7] & = e^*_8,  \\
			            [e_4,e_8] & = -e^*_7,\qquad & [e_5,e_6] & = e^*_4, \qquad & [e_5,e_7] & = e^*_2, \qquad & [e_6,e_8] & = e^*_3,  \\
			            [e_7,e_8] &= e^*_1 + e^*_4.
		            \end{alignat*}
	      \end{itemize}
\end{itemize}

However, despite there are an infinite number of non-isometrically isomorphic 2-step quadratic Lie algebras of dimension greater or equal than 18, for the 18-dimensional algebras we still have a classification based on seven families of trivectors, where each family depends on some parameters. This classification can be found in \cite{Vinberg_Elashvili_1988}. Here, the authors explain how every trivector can be decompose as a sum of a semisimple trivector and a nilpotent one. Details of this concepts can be found in \S 1 of the paper. That semisimple part is a linear combination of four special trivectors, and it is that specific combination in addition to the nilpotent part what defines each family. Despite this classification involves parameters $\lambda_i$, this does not affect our conversion to 2-step quadratic Lie algebras, and the procedure described in \eqref{eq:process_trivector_algebra} is still functional. We can see it in the following example.
\begin{example}
	Let us take a trivector $u$ in the sixth family, which decomposes as $u = p + e$ with $p = \lambda (123 + 456 + 789) \neq 0$, where $\lambda$ is determined up to multiplication by a sixth root of the unity, and $e$ is in table~\cite[Table 5]{Vinberg_Elashvili_1988}. For example, we consider $e = 147 + 158$. So our trivector is
	\begin{equation*}
		u = \lambda(123 + 456 + 789) + 147 + 158.
	\end{equation*}
	In this case the associated 18-dimensional Lie algebra $\fL$ is defined by products
	\begin{alignat*}{4}
		[e_1,e_2] & = \lambda e^*_3, \qquad  & [e_1,e_3] & = -\lambda e^*_2, \qquad & [e_1,e_4] & = e^*_7, \qquad        & [e_1,e_5] & = e^*_8,         \\
		[e_1,e_7] & = -e^*_4,\qquad          & [e_1,e_8] & = -e^*_5, \qquad         & [e_2,e_3] & = \lambda e^*_1,\qquad & [e_4,e_5] & = \lambda e^*_6, \\
		[e_4,e_6] & = -\lambda e^*_5, \qquad & [e_4,e_7] & = e^*_1,\qquad           & [e_5,e_6] & = \lambda e^*_4,\qquad & [e_5,e_8] & = e^*_	1,         \\
		[e_7,e_8] &= \lambda e^*_9,\qquad &[e_7,e_9] &= -\lambda e^*_8,\qquad &[e_8,e_9] &= \lambda e^*_7,
	\end{alignat*}
	when considering the basis $\{e_1, \ldots, e_9, e^*_1, \ldots, e^*_9\}$.
\end{example}
\begin{remark}
	The idea of a nilpotent or semisimple trivector does not affect the nilpotency of the algebra obtained from it. By construction, the quadratic algebra we obtain is always 2-step.
\end{remark}
\begin{remark}
	In case we are interested in greater dimensions, things start to get much more difficult as there are not a finite number of trivectors-algebras. We can get examples by applying Theorem \ref{thm:reducedFamily} or Lemma~\ref{lem:chain}. These results provides computational constructions based on non-degenerate $n$-quadratic matrices or chained one-dimensional double extensions.
\end{remark}

\section*{Funding}
The authors are partially funded by grant MTM2017-83506-C2-1-P of `Ministerio de Econom\'ia, Industria y Competitividad, Gobierno de Espa\~na' (Spain). J. Rold\'an-L\'opez is also supported by a predoctoral research grant FPI-2018 of `Universidad de La Rioja'.

\bibliographystyle{tfnlm}
\bibliography{bibliography}

\begin{thebibliography}{10}
\providecommand{\url}[1]{\normalfont{#1}}
\providecommand{\urlprefix}{Available from: }

\bibitem{Tsou_Walker_1957}
Tsou~ST, Walker~AG. {XIX}.-{M}etrisable {L}ie groups and algebras. Proc Sect A,
  Math phys sci. 1957;\hspace{0pt}64:290--304.

\bibitem{milnor1976curvatures}
Milnor~J. Curvatures of left invariant metrics on {L}ie groups. Advances in
  mathematics. 1976;\hspace{0pt}21(3):293--329.

\bibitem{medina1985groupes}
Medina~A. Groupes de {L}ie munis de m{\'e}triques bi-invariantes. Tohoku
  Mathematical Journal, Second Series. 1985;\hspace{0pt}37(4):405--421.

\bibitem{bordemann1997nondegenerate}
Bordemann~M. Nondegenerate invariant bilinear forms on nonassociative algebras.
  Acta Math Univ Comenianae. 1997;\hspace{0pt}66(2):151--201.

\bibitem{benito2017free}
Benito~P, de-la Concepci{\'o}n~D, Laliena~J. Free nilpotent and nilpotent
  quadratic {L}ie algebras. Linear Algebra and its Applications.
  2017;\hspace{0pt}519:296--326.

\bibitem{keith1984invariant}
Keith~VS. On invariant bilinear forms on finite-dimensional {L}ie algebras.
  Tulane University; 1984.

\bibitem{favre_santharoubane_1987}
Favre~G, Santharoubane~LJ. Symmetric, invariant, non-degenerate bilinear form
  on a {L}ie algebra. Journal of Algebra. 1987;\hspace{0pt}105(2):451--464.

\bibitem{medina1985algebres}
Medina~A, Revoy~P. Algebres de {L}ie et produit scalaire invariant. Annales
  scientifiques de l'{\'E}cole Normale Sup{\'e}rieure.
  1985;\hspace{0pt}18(3):553--561.

\bibitem{kac1990infinite}
Kac~VG. Infinite-dimensional {L}ie algebras. Cambridge university press; 1990.

\bibitem{kath2006metric}
Kath~I, Olbrich~M. Metric {L}ie algebras and quadratic extensions.
  Transformation groups. 2006;\hspace{0pt}11(1):87--131.

\bibitem{kath2009structure}
Kath~I, Olbrich~M. On the structure of pseudo-riemannian symmetric spaces.
  Transformation groups. 2009;\hspace{0pt}14(4):847--885.

\bibitem{tsou1962xi}
Tsou~ST. {XI}.-on the construction of metrisable {L}ie algebras. Proceedings of
  the Royal Society of Edinburgh Section A: Mathematics.
  1962;\hspace{0pt}66(2):116--127.

\bibitem{benito2019quadratic}
Benito~P, de-la Concepci{\'o}n~D, Rold{\'a}n-L{\'o}pez~J, et~al. Quadratic
  2-step {L}ie algebras: Computational algorithms and classification. Journal
  of Symbolic Computation. 2019;\hspace{0pt}94:70--89.

\bibitem{noui1997algebres}
Noui~L, Revoy~P. Algebres de {L}ie orthogonales et formes trilin{\'e}aires
  altern{\'e}es. Communications in algebra. 1997;\hspace{0pt}25(2):617--622.

\bibitem{Duong_2013}
Duong~MT. Two-step nilpotent quadratic {L}ie algebras and 8-dimensional
  non-commutative symmetric {N}ovikov algebras. Vietnam Jornal of Mathematics.
  2013;\hspace{0pt}41:135--148.

\bibitem{figueroa1996structure}
Figueroa-O’Farrill~JM, Stanciu~S. On the structure of symmetric self-dual
  {L}ie algebras. Journal of Mathematical Physics.
  1996;\hspace{0pt}37(8):4121--4134.

\bibitem{elman2008algebraic}
Elman~RS, Karpenko~N, Merkurjev~A. The algebraic and geometric theory of
  quadratic forms. Vol.~56. American Mathematical Soc.; 2008.

\bibitem{hofmann1986invariant}
Hofmann~KH, Keith~VS. Invariant quadratic forms on finite dimensional {L}ie
  algebras. Bulletin of the Australian Mathematical Society.
  1986;\hspace{0pt}33(1):21--36.

\bibitem{gauger_1973}
Gauger~MA. On the classification of metabelian {L}ie algebras. Transactions of
  the American Mathematical Society. 1973;\hspace{0pt}179:293--293.

\bibitem{fulton2013representation}
Fulton~W, Harris~J. Representation theory: a first course. Vol. 129. Springer
  Science \& Business Media; 2013.

\bibitem{cohen1988trilinear}
Cohen~AM, Helminck~AG. Trilinear alternating forms on a vector space of
  dimension 7. Communications in algebra. 1988;\hspace{0pt}16(1):1--25.

\bibitem{noui1994formes}
Noui~L, Revoy~P. Formes multilin{\'e}aires altern{\'e}es. Annales
  Math{\'e}matiques Blaise Pascal. 1994;\hspace{0pt}1(2):43--69.

\bibitem{Vinberg_Elashvili_1988}
Vinberg~{\`E}B, {\`E}lashvili~AG. Classification of {T}rivectors of a
  9-{D}imensional {S}pace. Sel Math Sov Birkhaeuser Verlag, Basel.
  1988;\hspace{0pt}7(1):63--58.

\end{thebibliography}

\end{document}